\documentclass[11pt]{amsart}
\usepackage{geometry}
\usepackage{graphicx}
\usepackage{amssymb}
\usepackage{epstopdf}
\usepackage{amsmath,amscd}
\usepackage{amsthm}
\usepackage{enumerate}
\usepackage{url,verbatim}
\usepackage{subfig}
\usepackage{enumitem}

\RequirePackage[colorlinks,citecolor=blue,urlcolor=blue]{hyperref}
\theoremstyle{plain}
\newtheorem{theorem}{Theorem}
\newtheorem{definition}[theorem]{Definition}
\newtheorem{lemma}[theorem]{Lemma}
\newtheorem{proposition}[theorem]{Proposition}
\newtheorem{corollary}[theorem]{Corollary}

\newtheorem{assumption}[theorem]{Assumption}
\newtheorem{remark}[theorem]{Remark}

\newcommand\ol{\overline}

\newcommand\RR{{\mathbb R}}

\newcommand\rank{\mathrm{rank}}

\renewcommand\ell{l}

\newcommand\CC{\mathbb{C}}

\newcommand\bW{\mathbf{W}}
\newcommand\bA{\mathbf{A}}

\newcounter{mycount}

\numberwithin{equation}{section}
\numberwithin{theorem}{section}
\numberwithin{figure}{section}

\title[Community Detection under Globally Dependent Gaussian Noise]{Exact Recovery for Community Detection under Globally Dependent Gaussian Noise}

\author{Zhongyang Li}
\address{Department of Mathematics,
University of Connecticut,
Storrs, Connecticut 06269-3009, USA}
\email{zhongyang.li@uconn.edu}
\urladdr{\url{https://mathzhongyangli.wordpress.com}}
\author{Sichen Yang}
\address{Department of Applied Mathematics and Statistics, Johns Hopkins University, Baltimore, Maryland, 21218, USA}
\email{syang114@jhu.edu}

\begin{document}

\begin{abstract}
We study exact recovery for community detection from observations
\[
\mathbf K_y=\mathbf A_y+\mathbf W,
\]
where the centered Gaussian noise is allowed to have an arbitrary covariance matrix
\[
\Sigma=\operatorname{Cov}(\operatorname{vec}\mathbf W)\in\mathbb R^{pn\times pn}.
\]
Thus \(\Sigma\) is the covariance of the full vectorized noise array, not merely the
common covariance of independent \(p\)-dimensional observations; in particular, noise
coordinates attached to different vertices may be correlated. This distinction is essential:
a row-wise whitening reduction applies only in the special case
\(\Sigma=I_n\otimes\Sigma_0\), whereas a general whitening by
\(\Sigma^{\dagger/2}\) mixes vertex coordinates and does not produce an ordinary
row-independent Gaussian mixture model.

We formulate the maximum likelihood estimator under possibly singular covariance by
writing the Gaussian likelihood on the support of the induced measure. The resulting MLE
is a constrained quadratic optimization problem involving the Moore--Penrose inverse.
We derive general sufficient conditions for exact recovery, and, under invertibility,
converse conditions based on asymptotically diagonal families of local Gaussian comparison statistics. We also isolate the
classical row-independent whitening case as a benchmark and then give cross-vertex
dependent examples showing how the general theory applies beyond that benchmark.
\end{abstract}

\maketitle

\section{Introduction}

Community detection concerns recovery of a latent partition from noisy observations. In the exact-recovery regime, one asks whether the true partition can be identified with probability tending to one as $n\to\infty$. This question has been studied extensively for the stochastic block model; see \cite{HLL83,MNS13,ABH15,EA18} and the algorithmic developments in \cite{FK13,LM14,AM15,AB16,AS18}. For Gaussian mixture formulations of community detection, exact recovery under independent Gaussian perturbations was studied in \cite{ZL19,ZL20}. Dependent models have also attracted attention, for instance in block-covariance clustering \cite{BGRV16,BGLRV20} and in the Ising block model \cite{BRS19}. The present paper studies exact recovery in a Gaussian mixture model with general dependent and heterogeneous Gaussian noise.

Specifically, \(n\) vertices are partitioned into \(k\ge2\) communities, and to each
vertex we attach a \(p\)-dimensional observation. We store the observations as a
\(p\times n\) matrix, with columns indexed by vertices:
\[
\mathbf K_y=\mathbf A_y+\mathbf W.
\]
The important point is that the covariance matrix in this paper is
\[
\Sigma=\operatorname{Cov}(\operatorname{vec}\mathbf W)\in\mathbb R^{pn\times pn}.
\]
It is therefore a covariance matrix for the entire noise array. We do not assume that
the columns of \(\mathbf W\) are independent, nor that \(\Sigma\) has the product form
\(I_n\otimes\Sigma_0\). Equivalently, for two distinct vertices \(j\ne l\), the cross-covariance
\[
\operatorname{Cov}(\mathbf W_{\cdot,j},\mathbf W_{\cdot,l})
\]
may be nonzero and heterogeneous.

This should be distinguished from the classical row-independent Gaussian mixture model.
If \(\Sigma=I_n\otimes\Sigma_0\), then one can whiten each vertex observation separately,
and the model reduces to an ordinary Gaussian mixture with transformed centers. That
benchmark case is discussed explicitly below. The focus of the present paper is the general
case in which whitening by \(\Sigma^{\dagger/2}\) is a global operation on
\(\operatorname{vec}\mathbf K_y\), mixes different vertices, and therefore does not preserve
the row-wise mixture structure.

A first issue is that once $\Sigma$ is allowed to be singular, the usual Gaussian density with respect to Lebesgue measure is no longer available. The correct likelihood must instead be written on the support of the induced Gaussian measure, which leads to a constrained optimization formulation of the MLE. This viewpoint also identifies the natural geometric quantity governing recoverability, namely the $\Sigma$-whitened squared separation
\[
L_\Sigma(x,y)
=
(\overrightarrow{\mathbf{A}_x-\mathbf{A}_y})^t \Sigma^\dagger \, \overrightarrow{\mathbf{A}_x-\mathbf{A}_y}.
\]
Throughout the paper, $L_\Sigma(x,y)$ plays the role of the effective signal-to-noise separation between two assignments. It controls both the Gaussian comparison bounds in the sufficient theory and the covariance structure of the local perturbation statistics in the converse theory.

Our main theoretical contribution is summarized by a compact sharp-threshold criterion. It identifies three deterministic ingredients that govern exact recovery for fixed $k$: global separation away from the true equivalence class, local one-step correction margins near the truth, and the covariance structure of the dangerous local comparison statistics. This criterion is the form used in the examples and makes explicit when the sufficient and necessary bounds match at a single scale $\Delta_n\sim 8\log n$.

The full technical results are stated at the beginning of Section~\ref{sect:p215} and proved in Sections~\ref{sect:p215}--\ref{sect:pm4}. For unknown community sizes, Theorem~\ref{p215} gives a sufficient condition based on a global/near-truth decomposition. For known community sizes, Theorem~\ref{m27} gives the corresponding size-constrained sufficient condition. Under the additional assumption that $\Sigma$ is invertible, Theorems~\ref{mm3} and~\ref{p31} provide converse results for families of local Gaussian comparison statistics with asymptotically diagonal covariance. Together, these results show how dependence enters through the geometry induced by $\Sigma^\dagger$ and through the covariance of the local perturbation field.

The general theory is illustrated in Section~\ref{sect:ex} through a sequence of
examples that separates the classical row-wise whitening benchmark from genuinely
cross-vertex dependent models. We first revisit a row-independent non-diagonal
block-covariance model, in which the columns of the observation matrix are independent
but the entries within each column are correlated. This benchmark shows that our
general likelihood criterion recovers the classical whitened threshold. We then turn to
models whose covariance is not of the product form \(I_n\otimes \Sigma_0\). In a
common-factor vertex-covariance model and in a more general precision-perturbation
model, the noise variables attached to different vertices are correlated, and whitening by
\(\Sigma^{-1/2}\) mixes vertex coordinates. Consequently, the transformed model is not
an ordinary row-independent Gaussian mixture with fixed row-wise centers. For these
cross-vertex dependent examples, we prove matching upper and lower bounds for exact
recovery of the MLE in the unknown-community-size setting. The thresholds are expressed
through a single deterministic separation scale \(\Delta_n\), which captures the effective
noise level in the relevant signal directions. These examples also demonstrate the
compact no-gap mechanism developed in Section~\ref{sect:mr}: global separation,
local one-step correction margins, and an asymptotically diagonal family of local
comparison statistics together yield a sharp threshold at \(\Delta_n\sim 8\log n\).

The paper is concerned with the statistical limit of the MLE rather than its efficient computation. In the i.i.d. Gaussian case, the likelihood-based objective becomes the classical $K$-means criterion once the unknown means are replaced by empirical within-cluster averages, and convex relaxations of $K$-means have been studied extensively; see, for example, \cite{PW07}. In the dependent setting, one may expect analogous relaxations based on the $\Sigma^{-1}$-weighted inner product when $\Sigma$ is invertible. We do not pursue that direction here, and instead focus on the benchmark exact-recovery boundary for the likelihood-based estimator.

The rest of the paper is organized as follows. Section~\ref{sect:mr} gives the streamlined model formulation, the support-aware likelihood, the row-wise whitening benchmark, and the compact sharp-threshold criterion used in the examples. Section~\ref{sect:p215} first collects the full technical assumptions and theorem statements, preserving the labels used in the proofs, and then proves the unknown-community-size sufficient theorem and Corollary~\ref{c212}. Section~\ref{sect:pm27} proves the known-size sufficient theorem. Sections~\ref{sect:pm3} and~\ref{sect:pm4}, under the additional assumption that $\Sigma$ is invertible, prove the corresponding converse results for unknown and known community sizes. Section~\ref{sect:ex} verifies the compact criterion in a row-wise benchmark, a cross-vertex common-factor covariance model, and a precision-perturbation model, thereby giving examples with matching upper and lower bounds.

\section{Main Results}\label{sect:mr}

This section defines the model, formulates the likelihood-based estimators, and states a
compact sharp-threshold criterion.  The purpose is to separate the three conceptual issues
that drive the paper: the likelihood geometry under possibly singular covariance, the
classical row-wise whitening benchmark, and the genuinely global nature of whitening when
noise is dependent across vertices.

\subsection{Model and notation}\label{subsec:model}

Let \([n]=\{1,\ldots,n\}\) be the set of vertices and let \([k]=\{1,\ldots,k\}\) be the set of
communities.  A community assignment is a map
\[
  x:[n]\to[k],
\]
and we write
\[
  \Omega:=\{x:[n]\to[k]\}
\]
for the set of all assignments.  The true assignment is denoted by \(y\in\Omega\).

For each \(x\in\Omega\), let \(\mathbf A_x\in\mathbb R^{p\times n}\) be the signal matrix.  In
the most general formulation considered in the paper, its entries are
\begin{equation}\label{eq:signal-general}
  (\mathbf A_x)_{i,j}=\theta(x,i,x(j)),\qquad i\in[p],\ j\in[n].
\end{equation}
This allows the mean attached to a vertex to depend on its label and, through \(x\), on the
global assignment.  In the standard finite-mixture special case one has
\((\mathbf A_x)_{\cdot,j}=\mu_{x(j)}\) for fixed centers
\(\mu_1,\ldots,\mu_k\in\mathbb R^p\).

The observation under the true assignment \(y\) is
\begin{equation}\label{eq:model}
  \mathbf K_y=\mathbf A_y+\mathbf W,
\end{equation}
where \(\mathbf W\in\mathbb R^{p\times n}\) is centered Gaussian.  We use column-wise
vectorization and write
\[
  \operatorname{vec}(\mathbf W)\in\mathbb R^{pn}.
\]
The covariance matrix in this paper is
\begin{equation}\label{eq:global-covariance}
  \Sigma=\operatorname{Cov}\{\operatorname{vec}(\mathbf W)\}\in\mathbb R^{pn\times pn}.
\end{equation}
Thus \(\Sigma\) is the covariance of the full vectorized noise array.  It is not assumed to be
of the product form \(I_n\otimes\Sigma_0\).  In particular, for two different vertices
\(j\ne \ell\), the cross-covariance
\[
  \operatorname{Cov}(\mathbf W_{\cdot,j},\mathbf W_{\cdot,\ell})
\]
may be nonzero and may vary with \(j,\ell\).

The equivalence class of the true assignment is
\begin{equation}\label{eq:equivalence-class}
  C(y):=\{x\in\Omega:\mathbf A_x=\mathbf A_y\}.
\end{equation}
Exact recovery means recovery of \(C(y)\), since assignments in the same class generate the
same signal matrix and are statistically indistinguishable.  For \(c\in(0,1)\), define the
balanced-assignment set
\begin{equation}\label{eq:omega-c}
  \Omega_c:=\left\{x\in\Omega:\min_{a\in[k]}|x^{-1}(a)|\ge cn\right\}.
\end{equation}
Most sharp-threshold statements below are stated for \(y\in\Omega_c\) and fixed \(k\).

\subsection{Likelihood under possibly singular covariance}\label{subsec:likelihood}

We allow \(\Sigma\) to be positive semidefinite and possibly singular.  Let
\[
  R_\Sigma:=\operatorname{Im}(\Sigma),
  \qquad
  \Pi_\Sigma:=\Sigma\Sigma^\dagger,
\]
where \(\Sigma^\dagger\) is the Moore--Penrose inverse and \(\Pi_\Sigma\) is the orthogonal
projection onto \(R_\Sigma\).  Conditional on the hypothesis \(y=x\), the observation is
supported on the affine subspace
\[
  \operatorname{vec}(\mathbf A_x)+R_\Sigma.
\]
Accordingly, the Gaussian likelihood has to be written with respect to the Lebesgue measure
on this affine support.  Up to a normalizing constant independent of \(x\), the negative log
likelihood is
\begin{equation}\label{eq:G-sigma}
G_\Sigma(x;\mathbf K)
:=
\begin{cases}
\bigl(\operatorname{vec}(\mathbf K-\mathbf A_x)\bigr)^\top
\Sigma^\dagger
\operatorname{vec}(\mathbf K-\mathbf A_x),
&
(I-\Pi_\Sigma)\operatorname{vec}(\mathbf K-\mathbf A_x)=0,\\[0.5em]
+\infty,
&
(I-\Pi_\Sigma)\operatorname{vec}(\mathbf K-\mathbf A_x)\ne0.
\end{cases}
\end{equation}
The unconstrained MLE is therefore
\begin{equation}\label{eq:mle-unknown-size}
  \widehat y\in\operatorname*{argmin}_{x\in\Omega}G_\Sigma(x;\mathbf K_y).
\end{equation}
If the community sizes \(n_a=|y^{-1}(a)|\) are known, the size-constrained MLE is
\begin{equation}\label{eq:mle-known-size}
  \widehat y_{\mathbf n}\in
  \operatorname*{argmin}_{x\in\Omega_{\mathbf n}}G_\Sigma(x;\mathbf K_y),
  \qquad
  \Omega_{\mathbf n}:=\{x\in\Omega:|x^{-1}(a)|=n_a\text{ for all }a\in[k]\}.
\end{equation}
For compatibility with the full technical statements below, we also write
\[
  \Omega_{n_1,\ldots,n_k}:=\Omega_{\mathbf n}.
\]

The basic deterministic separation between two assignments is the support-aware quantity
\begin{equation}\label{eq:D-sigma}
D_\Sigma(x,z):=
\begin{cases}
\bigl(\operatorname{vec}(\mathbf A_x-\mathbf A_z)\bigr)^\top
\Sigma^\dagger
\operatorname{vec}(\mathbf A_x-\mathbf A_z),
&
(I-\Pi_\Sigma)\operatorname{vec}(\mathbf A_x-\mathbf A_z)=0,\\[0.5em]
+\infty,
&
(I-\Pi_\Sigma)\operatorname{vec}(\mathbf A_x-\mathbf A_z)\ne0.
\end{cases}
\end{equation}
When \(\Sigma\) is invertible, this reduces to
\begin{equation}\label{eq:L-sigma}
  L_\Sigma(x,z)
  :=
  \bigl(\operatorname{vec}(\mathbf A_x-\mathbf A_z)\bigr)^\top
  \Sigma^{-1}
  \operatorname{vec}(\mathbf A_x-\mathbf A_z).
\end{equation}
For singular \(\Sigma\), the same notation \(L_\Sigma\) is used for the finite quadratic part
with \(\Sigma^\dagger\); the support-aware quantity \(D_\Sigma\) additionally records whether
an assignment is feasible in the null-space directions.

\begin{lemma}[Deterministic information in the null space]\label{lem:null-space}
Under the true assignment \(y\),
\[
  (I-\Pi_\Sigma)\operatorname{vec}(\mathbf K_y)
  =
  (I-\Pi_\Sigma)\operatorname{vec}(\mathbf A_y)
  \qquad\text{almost surely}.
\]
Consequently, if
\[
  (I-\Pi_\Sigma)\operatorname{vec}(\mathbf A_x)
  \ne
  (I-\Pi_\Sigma)\operatorname{vec}(\mathbf A_y),
\]
then \(x\) has zero likelihood under data generated from \(y\).  In particular, if the map
\[
  x\mapsto (I-\Pi_\Sigma)\operatorname{vec}(\mathbf A_x)
\]
separates \(C(y)\) from all other assignments, exact recovery is deterministic.
\end{lemma}

\begin{proof}
Since \(\operatorname{vec}(\mathbf W)\in R_\Sigma\) almost surely, its projection onto
\(R_\Sigma^\perp\) is zero.  Hence the null-space projection of the observation equals the
null-space projection of the signal.  The zero-likelihood statement follows from the support
condition in \eqref{eq:G-sigma}.
\end{proof}

\subsection{The row-wise whitening benchmark}\label{subsec:row-wise-benchmark}

The next proposition isolates the classical case in which whitening truly reduces the model
to an ordinary row-independent Gaussian mixture.  It is included as a benchmark and as a
point of comparison for the globally dependent models studied later.

\begin{proposition}[Row-wise whitening benchmark]\label{prop:row-wise-whitening}
Assume that the signal has fixed mixture centers,
\[
  (\mathbf A_x)_{\cdot,j}=\mu_{x(j)},
  \qquad
  \mu_1,\ldots,\mu_k\in\mathbb R^p,
\]
and assume that the noise covariance has the product form
\[
  \Sigma=I_n\otimes\Sigma_0
\]
for a positive semidefinite \(p\times p\) matrix \(\Sigma_0\).  Let
\[
  \Sigma_0=U_r\Lambda_rU_r^\top
\]
be the spectral decomposition on \(\operatorname{Im}(\Sigma_0)\), and let \(U_0\) be an
orthonormal basis of \(\ker(\Sigma_0)\).  Then the observation decomposes into a
deterministic null-space part
\[
  U_0^\top\mathbf K_{\cdot,j}=U_0^\top\mu_{y(j)}
\]
and an independent whitened part
\[
  \Lambda_r^{-1/2}U_r^\top\mathbf K_{\cdot,j}
  =
  \Lambda_r^{-1/2}U_r^\top\mu_{y(j)}+\xi_j,
  \qquad
  \xi_j\stackrel{\mathrm{i.i.d.}}{\sim}N(0,I_r).
\]
Therefore, if the null-space projections \(U_0^\top\mu_a\) distinguish the communities,
recovery is deterministic in those directions.  Otherwise the model is exactly the standard
isotropic row-independent Gaussian mixture with transformed centers
\[
  \widetilde\mu_a:=\Lambda_r^{-1/2}U_r^\top\mu_a,
  \qquad a\in[k].
\]
\end{proposition}

\begin{proof}
Since \(\Sigma=I_n\otimes\Sigma_0\), the columns of \(\mathbf W\) are independent
\(N(0,\Sigma_0)\) vectors. The null-space identity follows from
\(U_0^\top\mathbf W_{\cdot,j}=0\) almost surely. On the range of \(\Sigma_0\),
multiplication by \(\Lambda_r^{-1/2}U_r^\top\) gives
\[
  \Lambda_r^{-1/2}U_r^\top\mathbf W_{\cdot,j}\sim N(0,I_r),
\]
independently across \(j\). This proves the decomposition and the reduction to the
standard isotropic row-independent Gaussian mixture.
\end{proof}

\begin{remark}[Why global whitening is different]\label{rem:global-whitening}
For a general covariance matrix \(\Sigma\in\mathbb R^{pn\times pn}\), whitening is the global
operation
\[
  \Sigma^{\dagger/2}\operatorname{vec}(\mathbf K_y)
  =
  \Sigma^{\dagger/2}\operatorname{vec}(\mathbf A_y)+\xi.
\]
This always gives standard Gaussian noise on \(\operatorname{Im}(\Sigma)\), but it generally
mixes coordinates belonging to different vertices.  Hence the transformed mean vector is a
codeword indexed by the whole assignment \(y\), not a collection of independent observations
whose \(j\)-th mean depends only on \(y(j)\).  The row-wise Gaussian-mixture reduction in
Proposition~\ref{prop:row-wise-whitening} is valid only in the special product-covariance
case.
\end{remark}

The following elementary lemma makes this distinction explicit for the Kronecker-type
examples used in Section~\ref{sect:ex}.

\begin{lemma}[When whitening preserves a row-wise mixture form]\label{lem:non-reduction}
Assume
\[
  (\mathbf A_x)_{\cdot,j}=s_n\mu_{x(j)},
  \qquad j\in[n],
\]
where \(s_n\ne0\) and at least two centers are distinct.  Let \(D\in\mathbb R^{q\times p}\) and
\(B_n\in\mathbb R^{n\times n}\) be deterministic matrices, and define
\[
  \widetilde{\mathbf A}_x:=D\mathbf A_xB_n.
\]
Suppose that for every assignment \(x\in\Omega\), the transformed signal still has a
row-wise mixture representation with the same vertex index, namely there exist vectors
\(\nu_1,\ldots,\nu_k\in\mathbb R^q\), independent of \(x\), such that
\[
  (\widetilde{\mathbf A}_x)_{\cdot,i}=\nu_{x(i)},
  \qquad i\in[n].
\]
If \(D(\mu_a-\mu_b)\ne0\) for some pair \(a\ne b\), then \(B_n\) must be diagonal.  If the
output vertices are allowed to be relabeled by a fixed permutation, then \(B_n\) must be
monomial.
\end{lemma}

\begin{proof}
Fix \(i\in[n]\).  The \(i\)-th transformed column is
\[
  (\widetilde{\mathbf A}_x)_{\cdot,i}
  =
  s_nD\sum_{j=1}^n(B_n)_{j i}\mu_{x(j)}.
\]
If \(j\ne i\), keep \(x(i)\) fixed and change only \(x(j)\) from \(a\) to \(b\), where
\(D(\mu_a-\mu_b)\ne0\).  The assumed row-wise representation says that the \(i\)-th
transformed column cannot change, since \(x(i)\) is unchanged.  Therefore
\[
  s_n(B_n)_{j i}D(\mu_a-\mu_b)=0,
\]
which forces \((B_n)_{j i}=0\).  This holds for all \(j\ne i\), so \(B_n\) is diagonal.  The
monomial version follows by applying the same argument after the fixed output permutation.
\end{proof}

\subsection{A compact sharp-threshold criterion}\label{subsec:compact-criterion}

The full technical statements collected at the beginning of Section~\ref{sect:p215} and proved in Sections~\ref{sect:p215}--\ref{sect:pm4} allow more general
bookkeeping, including size-constrained recovery.  The following compact criterion is the
form used in the main examples.  It keeps only the three deterministic ingredients that
matter for a sharp threshold:

\begin{enumerate}
\item a global separation away from the true equivalence class;
\item a local one-step correction margin near the truth;
\item a large family of local alternatives whose comparison statistics are asymptotically
      diagonal.
\end{enumerate}

For \(x,z\in\Omega\), let
\[
  d_{\mathrm H}(x,z):=|\{i\in[n]:x(i)\ne z(i)\}|,
  \qquad
  d_{\mathrm H}(x,C(y)):=\min_{z\in C(y)}d_{\mathrm H}(x,z).
\]
We say that \(x^+\) is a one-step correction of \(x\) toward \(C(y)\) if there exists
\(z\in C(y)\) such that
\[
  d_{\mathrm H}(x,z)=d_{\mathrm H}(x,C(y))
\]
and \(x^+\) is obtained from \(x\) by changing exactly one coordinate \(v\) with
\(x(v)\ne z(v)\) to the value \(z(v)\).  Thus
\[
  d_{\mathrm H}(x^+,C(y))=d_{\mathrm H}(x,C(y))-1.
\]

\begin{proposition}[Compact sharp-threshold criterion]\label{prop:compact-sharp-threshold}
Assume that \(k\) is fixed and that \(y\in\Omega_c\) for some fixed \(c\in(0,1)\). Assume also that the equivalence class of the truth is uniformly bounded,
\[
  |C(y)|\le M_k,
\]
where \(M_k<\infty\) depends only on \(k\). Let
\(\Delta_n>0\) be a deterministic sequence.

Assume that for every sufficiently small \(\varepsilon>0\), the following two upper-bound
conditions hold.

\medskip
\noindent
\textup{(U1) Global separation.}
There exists a constant \(b_\varepsilon>0\) such that
\[
  \min_{\substack{x\in\Omega:\ d_{\mathrm H}(x,C(y))>\varepsilon n}}
  D_\Sigma(x,y)
  \ge
  b_\varepsilon n\Delta_n
\]
for all sufficiently large \(n\).

\medskip
\noindent
\textup{(U2) Local correction margin.}
There exists a number \(a_\varepsilon\ge0\), with
\[
  \lim_{\varepsilon\downarrow0}\limsup_{n\to\infty}a_\varepsilon=0,
\]
such that, uniformly over all
\[
  x\notin C(y),
  \qquad
  d_{\mathrm H}(x,C(y))\le \varepsilon n,
  \qquad
  D_\Sigma(x,y)<\infty,
\]
there exists a one-step correction \(x^+\) of \(x\) toward \(C(y)\), with
\[
  D_\Sigma(x^+,y)<\infty,
\]
satisfying
\[
  D_\Sigma(x,y)-D_\Sigma(x^+,y)
  \ge
  (1-a_\varepsilon-o(1))\Delta_n.
\]

Then, for every fixed \(\delta>0\),
\[
  \Delta_n\ge(8+\delta)\log n
  \quad\Longrightarrow\quad
  \mathbb P(\widehat y\in C(y))\to1.
\]

Assume in addition that \(\Sigma\) is invertible and that the following two lower-bound
conditions hold.

\medskip
\noindent
\textup{(L1) Local alternatives at scale \(\Delta_n\).}
There exists a set \(H_n\subset[n]\) with
\[
  \log|H_n|=(1+o(1))\log n
\]
and, for each \(a\in H_n\), an assignment \(y^{(a)}\notin C(y)\) differing from \(y\) at
exactly one vertex such that
\[
  L_\Sigma(y^{(a)},y)=\Delta_n(1+o(1))
\]
uniformly over \(a\in H_n\).

\medskip
\noindent
\textup{(L2) Asymptotically diagonal local covariance.}
For \(a\in H_n\), define
\[
  \eta_a
  :=
  \frac{
  2\bigl(\operatorname{vec}(\mathbf A_{y^{(a)}}-\mathbf A_y)\bigr)^\top
  \Sigma^{-1}\operatorname{vec}(\mathbf W)
  }{
  L_\Sigma(y^{(a)},y)
  }.
\]
Let \(\Phi_{H_n}\) be the covariance matrix of \(\{\eta_a:a\in H_n\}\).  Assume that
\[
  \Phi_{H_n}=\frac{4}{\Delta_n}(I+R_n),
  \qquad
  \|R_n\|_{\mathrm{op}}\to0.
\]

Then, for every fixed \(\delta>0\),
\[
  \Delta_n\le(8-\delta)\log n
  \quad\Longrightarrow\quad
  \mathbb P(\widehat y\in C(y))\to0.
\]
Consequently, whenever \textup{(U1)--(U2)} and \textup{(L1)--(L2)} hold with the same scale
\(\Delta_n\), the exact-recovery threshold of the MLE is sharp at
\[
  \Delta_n\sim8\log n.
\]
\end{proposition}

\begin{proof}
The sufficient implication is the streamlined energy--entropy argument underlying
Theorem~\ref{p215}.  For any fixed alternative whose null-space projection is incompatible
with the true signal, the likelihood of that alternative is zero by the support condition in
\eqref{eq:G-sigma}.  For the remaining alternatives, the Gaussian comparison bound gives
\[
  \mathbb P_y\{G_\Sigma(x;\mathbf K_y)\le G_\Sigma(y;\mathbf K_y)\}
  \le \exp\{-D_\Sigma(x,y)/8\}.
\]
Condition \textup{(U1)} controls the union bound over assignments satisfying
\(d_{\mathrm H}(x,C(y))>\varepsilon n\), since there are at most \(k^n\) such assignments and
\(D_\Sigma(x,y)\ge b_\varepsilon n\Delta_n\).  Condition \textup{(U2)} controls the near-truth
assignments with finite support-aware separation.  Alternatives with \(D_\Sigma(x,y)=+\infty\)
have zero likelihood under data generated from \(y\) and do not contribute to the union
bound. Repeatedly applying one-step corrections, each of which remains in the finite-separation set by \textup{(U2)}, gives, for an assignment at distance
\(h\le\varepsilon n\) from \(C(y)\) and with \(D_\Sigma(x,y)<\infty\),
\[
  D_\Sigma(x,y)\ge h(1-a_\varepsilon-o(1))\Delta_n.
\]
The number of assignments at Hamming distance \(h\) from \(C(y)\) is at most
\[
  |C(y)|\binom nh(k-1)^h\le M_k(nk)^h.
\]
Hence the near-truth contribution is bounded by a geometric series with ratio
\[
  nk\exp\{-(1-a_\varepsilon-o(1))\Delta_n/8\}.
\]
If \(\Delta_n\ge(8+\delta)\log n\), choose \(\varepsilon\) sufficiently small so that
\(a_\varepsilon\) is small.  The global and local error probabilities then both tend to zero,
which proves \(\mathbb P(\widehat y\in C(y))\to1\).

For the converse, assume \(\Sigma\) is invertible and \textup{(L1)--(L2)} hold.  For each
local alternative \(y^{(a)}\), the likelihood comparison can be written as
\[
  G_\Sigma(y^{(a)};\mathbf K_y)-G_\Sigma(y;\mathbf K_y)
  =L_\Sigma(y^{(a)},y)\{1-\eta_a\}.
\]
Thus exact recovery fails if \(\max_{a\in H_n}\eta_a\ge1\).  By \textup{(L2)} and
Lemma~\ref{lem:gaussian-max-asymp-diagonal}, applied with \(\sigma_N^2=4/\Delta_n\),
\[
  \max_{a\in H_n}\eta_a
  \ge
  (1-o_{\mathbb P}(1))\sqrt{\frac{8\log |H_n|}{\Delta_n}}.
\]
Since \(\log |H_n|=(1+o(1))\log n\) and \(\Delta_n\le(8-\delta)\log n\), the right-hand side
is larger than \(1\) with probability tending to one. Therefore
\(\mathbb P(\widehat y\in C(y))\to0\).
\end{proof}

\begin{remark}[Interpretation of the compact criterion]\label{rem:compact-interpretation}
Condition \textup{(U1)} controls the entropy of globally wrong assignments: there are at
most \(k^n\) of them, but each has likelihood comparison cost of order \(n\Delta_n\).  Condition
\textup{(U2)} controls the near-truth alternatives by an energy--entropy argument along a
path of single-vertex corrections.  A Hamming-distance \(h\) alternative has cost at least
approximately \(h\Delta_n\), while the number of such alternatives is at most of order
\((nk)^h\).  This gives the sufficient threshold \(\Delta_n>8\log n\).  Conditions
\textup{(L1)--(L2)} identify a family of \(n^{1+o(1)}\) dangerous one-vertex alternatives whose
Gaussian comparison statistics have variance \(4/\Delta_n\) and are asymptotically
decoupled; their maximum crosses one when \(\Delta_n<8\log n\).
\end{remark}

\subsection{Known community sizes}\label{subsec:known-size}

The size-constrained MLE \(\widehat y_{\mathbf n}\) in \eqref{eq:mle-known-size} is analyzed by
the same mechanism, except that one-step corrections must preserve the prescribed community
sizes.  Thus the local moves are swaps, or more generally short cycles of relabelings, rather
than single-vertex corrections.  The full size-constrained statement is stated in Section~\ref{sect:p215} and proved in Section~\ref{sect:pm27}.
Informally, the conditions are the following analogues of Proposition~\ref{prop:compact-sharp-threshold}:

\begin{enumerate}
\item globally separated size-preserving assignments have
      \(D_\Sigma(x,y)\ge b_\varepsilon n\Delta_n\);
\item every near-truth size-preserving alternative admits a size-preserving correction move
      that decreases \(D_\Sigma(\cdot,y)\) by at least \((1-o(1))\Delta_n\) per corrected vertex;
\item for the converse, there are \(n^{1+o(1)}\) size-preserving local alternatives whose
      comparison statistics have covariance \((4/\Delta_n)(I+o_{\mathrm{op}}(1))\).
\end{enumerate}

Under these conditions the same threshold \(\Delta_n\sim8\log n\) holds for the
size-constrained MLE.

\subsection{Relation to the full technical results}\label{subsec:full-results-map}

Proposition~\ref{prop:compact-sharp-threshold} is the main criterion used in the examples.
Section~\ref{sect:p215} begins by collecting the full versions of the sufficient and converse theorems.
Those results keep track of overlap tables, canonical label alignments, and size constraints, and their proofs are given in Sections~\ref{sect:p215}--\ref{sect:pm4}.
They are useful in models where the equivalence class \(C(y)\) is nontrivial or where the
known-size constraint forces corrections to occur through multi-vertex cycles.  The compact
criterion above is obtained from the full theorems by replacing that bookkeeping with the
Hamming-distance formulation \(d_{\mathrm H}(x,C(y))\) and by collecting all local increments at
a single deterministic scale \(\Delta_n\).

\section{Full Technical Statements and Exact Recovery with Unknown Community Sizes}
\label{sect:p215}

\subsection{Full technical statements used in the proofs}\label{subsec:full-technical-statements}

The compact criterion in Section~\ref{sect:mr} is the form used in the examples.  For the
proofs in Sections~\ref{sect:p215}--\ref{sect:pm4}, we now record the full technical
statements.  These statements retain the overlap-table bookkeeping, the singular-covariance
likelihood formulation, and the known-size constraints needed for the most general results.
The labels of the original theorem statements are preserved.

The full sufficient theorems below are stated in terms of the range-space quadratic
quantity $L_\Sigma$. They should be read as sufficient conditions on the stochastic part
of the likelihood. Candidates separated from the truth in the null-space directions have
zero likelihood under data generated from $y$, as explained in Lemma~\ref{lem:null-space}
and Lemma~\ref{le24}. The compact criterion in Section~\ref{subsec:compact-criterion}
records this support information explicitly through $D_\Sigma$.

\begin{definition}\label{da1}Let $A$ be an $n\times n$ square matrix. The pseudo-determinant of $A$ is defined by
\begin{eqnarray*}
\det^* A=\lim_{\alpha\rightarrow 0}\frac{\det (A+\alpha I)}{\alpha^{n-\mathrm{rank}(A)}}
\end{eqnarray*}
\end{definition}

Note that when $A$ is invertible, we have
\begin{eqnarray*}
\det A=\det^* A
\end{eqnarray*}

\begin{definition}\label{da2}Let $A\in \RR^{m\times n}$ be an $m\times n$ matrix with real entries. If $A^+\in \RR^{n\times m}$ satisfies all the following conditions
\begin{enumerate}
\item $AA^+A=A$; and
\item $A^+AA^+=A^+$; and
\item $(AA^+)^t=AA^+$; and
\item $(A^+A)^t=A^+A$
\end{enumerate}
then $A^+$ is called a Moore-Penrose inverse of $A$. The Moore-Penrose inverse of $A\in \RR^{m\times n}$ exists and is unique; see \cite{PR71,JM78}.
\end{definition}

Again if $A$ is nonsingular (i.e. invertible), we have
\begin{eqnarray*}
A^+=A^{-1},
\end{eqnarray*}
where $A^{-1}$ is the regular inverse of $A$.

\begin{proposition}\label{pa3}(Explicit construction of Moore-Penrose inverse for a symmetric matrix) Let $A\in \RR^{m\times m}$ be an $m\times m$ symmetric matrix with real entries such that the rank of $A$ is $r$. Assume 
\begin{eqnarray*}
A=PDP^{t};
\end{eqnarray*}
where 
\begin{itemize}
\item $P\in \RR^{m\times m}$ is an orthogonal matrix and $P^t$ is its transpose; and
\item $D\in \RR^{m\times m}$, such that the top left $r\times r$ block of $D$ is a diagonal matrix with diagonal entries $\lambda_1,\ldots,\lambda_r$ (eigenvalues of $A$); and all the other entries of $D$ are 0. More precisely, let
\begin{eqnarray*}
D_r=\mathrm{diag}(\lambda_1,\ldots,\lambda_r);
\end{eqnarray*}
then
\begin{eqnarray*}
D=\left(\begin{array}{cc}D_r&0\\0&0\end{array}\right).
\end{eqnarray*}
\end{itemize}
Assume 
\begin{eqnarray*}
P=(P_r,\overline{P}_r);
\end{eqnarray*}
where $P_r\in \RR^{m\times r}$ and $\overline{P}_r\in \RR^{m\times (m-r)}$. Then 
\begin{eqnarray*}
A^+=P_r D_r^{-1}(P_r)^t
\end{eqnarray*}
\end{proposition}

For a general matrix $A\in\CC^{m\times n}$, we may construct its Moore-Penrose inverse similarly by using the singular value decomposition.

\begin{proposition}\label{pa4}Let $A\in \RR^{m\times m}$ be a symmetric positive definite matrix. Let $\lambda_1$ be the maximal eigenvalue of $A$; let $\mu_0$ be the minimum of all the nonzero eigenvalues of $A^+$. Then
\begin{eqnarray*}
\mu_0\lambda_1=1.
\end{eqnarray*}
\end{proposition}
\begin{proof}Check directly from Definition \ref{da2}.
\end{proof}

\begin{definition}\label{da3}Let $A\in \CC^{n\times n}$ be a positive semi-definite matrix, then there is exactly one positive semi-definite matrix $B$ such that
\begin{eqnarray*}
A=B^*B
\end{eqnarray*}
where $B^*$ is the conjugate transpose of $B$. Then we define
\begin{eqnarray*}
A^{\frac{1}{2}}:=B.
\end{eqnarray*}
\end{definition}


\medskip
\noindent\textbf{Gaussian likelihood when $\Sigma$ may be singular.}
Let $d:=pn$ and consider $\mathbf{K}_y,\mathbf{A}_x,\mathbf{W}$ as vectors in $\mathbb R^{d}$ (e.g.\ through $\mathrm{vec}(\cdot)$).
Assume $\Sigma\succeq 0$ and let $r:=\rank(\Sigma)$.
Set the noise support (range space)
\[
\mathcal R_\Sigma:=\Im(\Sigma)=\Im(\Sigma^{1/2})\subset \mathbb R^{d},
\]
and let $\Pi_\Sigma:=\Sigma\Sigma^\dagger$ be the orthogonal projector onto $\mathcal R_\Sigma$.

Fix any matrix $U_r\in\mathbb R^{d\times r}$ whose columns form an orthonormal basis of
$\mathcal R_\Sigma$.
We define the $r$-dimensional Lebesgue measure on $\mathcal R_\Sigma$ by
\[
\lambda_{\mathcal R_\Sigma}(B)
:= \lambda_r\bigl(\{z\in\mathbb R^r:\ U_r z\in B\}\bigr),
\qquad B\subset \mathcal R_\Sigma,
\]
where $\lambda_r$ denotes Lebesgue measure on $\mathbb R^r$.
For each $x\in\Omega$, define the translated reference measure on the affine space
$\mathbf{A}_x+\mathcal R_\Sigma$ by
\[
\lambda_x(B):=\lambda_{\mathcal R_\Sigma}(B-\mathbf{A}_x),
\qquad B\subset \mathbf{A}_x+\mathcal R_\Sigma.
\]
(Equivalently, $\lambda_x$ is the $r$-dimensional Hausdorff/Lebesgue measure on
$\mathbf{A}_x+\mathcal R_\Sigma$ induced by the Euclidean structure.)

\medskip
\noindent\textbf{Density as a Radon--Nikodym derivative.}
Conditional on the hypothesis ``$y=x$'', the observation $\mathbf{K}_y$ is supported on
$\mathbf{A}_x+\mathcal R_\Sigma$ and admits the Radon--Nikodym derivative with respect to $\lambda_x$:
\begin{align}\label{eq:likelihood-singular}
&\frac{d\mathbb P(\mathbf{K}_y\in\cdot\mid y=x)}{d\lambda_x}(\mathbf{K}_y)\\
&=
\frac{1}{\sqrt{(2\pi)^{r}\det{}^{*}\Sigma}}\,
\exp\!\left(
-\frac12\sum_{i,k\in[p]}\sum_{j,l\in[n]}
(\mathbf{K}_y-\mathbf{A}_x)_{i,j}\,(\Sigma^\dagger)_{i,j;k,l}\,(\mathbf{K}_y-\mathbf{A}_x)_{k,l}
\right)\,
\mathbf 1_{\{(I-\Pi_\Sigma)\,\overrightarrow{(\mathbf{K}_y-\mathbf{A}_x)}=0\}},
\end{align}
where $\det{}^{*}\Sigma$ is the pseudo-determinant (product of the nonzero eigenvalues),
and $\overrightarrow{(\cdot)}$ denotes vectorization from $\mathbb R^{p\times n}$ to $\mathbb R^{pn}$.
In particular, if $\mathbf{K}_y\notin \mathbf{A}_x+\mathcal R_\Sigma$ (equivalently $(I-\Pi_\Sigma)\overrightarrow{(\mathbf{K}_y-\mathbf{A}_x)}\neq0$),
then the likelihood equals $0$.

\medskip
\begin{lemma}[MLE as a constrained quadratic minimization for possibly singular $\Sigma$]\label{le24}
Given an observation $\mathbf{K}_y$, define for each $x\in\Omega$
\[
G_\Sigma(x;\mathbf{K}_y):=
\begin{cases}
\displaystyle \sum_{i,k\in[p]}\sum_{j,l\in[n]}
(\mathbf{K}_y-\mathbf{A}_x)_{i,j}\,(\Sigma^\dagger)_{i,j;k,l}\,(\mathbf{K}_y-\mathbf{A}_x)_{k,l},
& \text{if } (I-\Pi_\Sigma)\overrightarrow{(\mathbf{K}_y-\mathbf{A}_x)}=0,\\[1.2ex]
+\infty, & \text{otherwise}.
\end{cases}
\]
Let
\begin{eqnarray}
\hat{y}:=\mathrm{argmin}_{x\in \Omega}G_\Sigma(x;\mathbf{K}_y)
\label{dhy}
\end{eqnarray}
and
\begin{eqnarray}
\check{y}:=\mathrm{argmin}_{x\in \Omega_{n_1,\ldots,n_k}}G_\Sigma(x;\mathbf{K}_y)\label{dcy2}
\end{eqnarray}

Then the MLE over $\Omega$ (resp.\ over $\Omega_{n_1,\dots,n_k}$) is $\hat{y}$ (resp.\ $\check{y}$).
\end{lemma}

\begin{proof}
By definition, the conditional law of $\overrightarrow{\mathbf{W}}$ is a centered Gaussian measure on
$\mathbb R^{pn}$ with covariance $\Sigma$, hence it is supported on $\mathcal R_\Sigma$ and
can be represented as $\overrightarrow{\mathbf{W}}=U_r\Lambda_r^{1/2}Z$ with $Z\sim N(0,I_r)$ and
$\Sigma=U_r\Lambda_r U_r^\top$.
Therefore, conditional on $y=x$, we have $\overrightarrow{\mathbf{K}_y}=\overrightarrow{\mathbf{A}_x}+U_r\Lambda_r^{1/2}Z$,
so $\overrightarrow{\mathbf{K}_y}\in \overrightarrow{\mathbf{A}_x}+\mathcal R_\Sigma$ almost surely and the Radon--Nikodym
derivative \eqref{eq:likelihood-singular} holds with respect to the reference measure
$\lambda_x$ defined above.
Since the prefactor $(2\pi)^{-r/2}(\det{}^*\Sigma)^{-1/2}$ is independent of $x$, maximizing
the likelihood is equivalent to minimizing the exponent, with the convention that the
likelihood is $0$ (i.e.\ the negative log-likelihood is $+\infty$) when
$\overrightarrow{\mathbf{K}_y}\notin \overrightarrow{\mathbf{A}_x}+\mathcal R_\Sigma$.
\end{proof}

\begin{remark}
\label{rem:feasible-set}
For a given observation $K_y$, define
\[
\mathcal F(\mathbf{K}_y):=
\Bigl\{
x\in\Omega:\ (I-\Pi_\Sigma)\overrightarrow{(\mathbf{K}_y-\mathbf{A}_x)}=0
\Bigr\}.
\]
Then, by Lemma~\ref{le24}, the MLE over $\Omega$ is the minimizer of $G_\Sigma(\cdot;\mathbf{K}_y)$
over $\mathcal F(\mathbf{K}_y)$, and similarly the MLE over $\Omega_{n_1,\ldots,n_k}$ is the
minimizer of $G_\Sigma(\cdot;\mathbf{K}_y)$ over
$\mathcal F(\mathbf{K}_y)\cap \Omega_{n_1,\ldots,n_k}$.
In the sufficient proofs below, we enlarge the error event by comparing the unrestricted
quadratic form $f$ over all candidates in $\Omega$ (or in $\Omega_{n_1,\ldots,n_k}$).
This enlargement is legitimate because every infeasible candidate has likelihood zero and
therefore cannot beat the true assignment under $G_\Sigma$.
\end{remark}

\subsection{Necessary and Sufficient Conditions for the Exact Recovery of MLE}

The main results of the paper are related to the necessary and sufficient conditions for the exact recovery of MLE. Before stating these conditions, we first introduce a few definitions. Definition \ref{df23} describes the sample space for community assignment mappings in which we implement the MLE. Instead of running the MLE among all possible community assignment mappings, we make a natural regularity assumption requiring that the number of vertices in each community is at least $cn$ for $c\in(0,1)$.

\begin{definition}\label{df23}
For each real number $c\in(0,1)$, let
\begin{eqnarray*}
\Omega_c:=\left\{x\in \Omega: \frac{|x^{-1}(i)|}{\sum_{j\in[k]}|x^{-1}(j)|}\geq c,\ \forall i\in[k]\right\},
\end{eqnarray*}
i.e. $\Omega_c$ consists of all community assignment mappings such that the ratio of the number of vertices in each community to the total number of vertices is at least $c$.
\end{definition}

An important observation is that the MLE can never distinguish two community assignment mappings obtained from each other by a composition with a $\theta$-preserving permutation of communities; see Lemma \ref{lfe}. Hence, the best thing one can expect from the MLE is to recover the community assignment mapping up to equivalence defined by a composition with a $\theta$-preserving permutation of communities.

\begin{definition}\label{dfeq}
For $x\in \Omega$, let $C(x)$ consist of all the $x'\in \Omega$ such that $x'$ can be obtained from $x$ by a $\theta$-preserving bijection of communities.  More precisely, $x'\in C(x)\subset \Omega$ if and only if the following conditions hold 
\begin{enumerate}
\item for $i,j\in[n]$, $x(i)=x(j)$ if and only if $x'(i)=x'(j)$; and
\item for $i\in[p]$ and $j\in[n]$, $\theta(x,i,x(j))=\theta(x',i,x'(j))$.
\end{enumerate}
Note that condition (1) above is equivalent to saying that there is a bijection $\eta:[k]\rightarrow [k]$, such that
\begin{eqnarray*}
x=\eta\circ x'
\end{eqnarray*}
where $\circ$ denotes the composition of two mappings; (2) says that the bijection $\eta$ is $\theta$-preserving.

We define an equivalence relation on $\Omega$ as follows: we say that $x,z\in \Omega$ are equivalent if and only if $x\in C(z)$. Let $\ol{\Omega}$ be the set of all equivalence classes in $\Omega$.
More precisely,
\begin{eqnarray*}
\ol{\Omega}:=\{C(x):x\in \Omega\}.
\end{eqnarray*}
\end{definition}

We also assume that $\theta$ satisfies the following assumption.

\begin{assumption}\label{ap24}Let $x,z\in \Omega$. If for any $i\in[p]$ and $j\in[n]$, 
\begin{eqnarray}
\theta(x,i,x(j))=\theta(z,i,z(j));\label{sxze}
\end{eqnarray}
then $x\in C(z)$.
\end{assumption}

Assumption \ref{ap24} actually says that for two community assignment mappings $x$ and $z$, if they are not equivalent, then $\theta\circ x$ and $\theta \circ z$ are different. In other words, it assumes that $\theta$ can distinguish different equivalence classes in $\Omega$. 

Under Assumption~\ref{ap24}, the equivalence class in Definition~\ref{dfeq}
coincides with the signal-equivalence class used in Section~\ref{subsec:model}. Indeed,
if \(z\in C(x)\), then \(\mathbf A_z=\mathbf A_x\) by Definition~\ref{dfeq}; conversely,
if \(\mathbf A_z=\mathbf A_x\), then Assumption~\ref{ap24} implies \(z\in C(x)\).

Define
\begin{eqnarray}
L_{\Sigma}(x,y):=\sum_{i,k\in[p];j,l\in[n]}\left(\bA_x-\bA_y\right)_{i,j}(\Sigma^{\dagger})_{i,j;k,l}
\left(\bA_x-\bA_y\right)_{k,l}\label{dls}
\end{eqnarray}

The scalar $L_\Sigma(x,y)$ is the $\Sigma$-whitened squared separation between $\mathbf{A}_x$
and $\mathbf{A}_y$ (with $\Sigma^\dagger$ handling possible degeneracy). Its role is that it directly
controls the MLE comparison: as shown in \eqref{fxmy}, the objective gap $f(x)-f(y)$ is
Gaussian with mean $L_\Sigma(x,y)$ and variance $4L_\Sigma(x,y)$, and hence
\[
\Pr\bigl(f(x)\le f(y)\bigr)\le \exp\!\left(-\frac{L_\Sigma(x,y)}{8}\right),
\]
which is the basic bound used in the union bound and the recovery condition~(\ref{ld1}).

\begin{definition}\label{df24}
For $i,j\in[k]$ and $x,z\in\Omega$, let $t_{i,j}(x,z)$ be the number of vertices in $[n]$ that are in the community $i$ under mapping $x$ and in the community $j$ under mapping $z$.
More precisely, $t_{i,j}(x,z)$ is a nonnegative integer given by
\begin{eqnarray*}
t_{i,j}(x,z)=|x^{-1}(i)\cap z^{-1}(j)|.
\end{eqnarray*}
satisfying 
\begin{eqnarray}
\sum_{j\in[k]}t_{i,j}(x,z)=n_i(x);\qquad \sum_{i\in[k]}t_{i,j}(x,z)=n_j(z);\label{tn}
\end{eqnarray}
Define
\begin{align*}
    S_{i,j}(x,z):=x^{-1}(i)\cap z^{-1}(j).
\end{align*}
\end{definition}

\begin{definition}\label{dfb}
Define a set 
 \begin{eqnarray}
 \mathcal{B}:=\left\{(t_{1,1},t_{1,2},\ldots,t_{k,k})\in\{0,1,2,\ldots,n\}^{k^2}:
 \sum_{i=1}^{k}\sum_{j=1}^{k}t_{i,j}=n\right\}.\label{dsb}
 \end{eqnarray}
 For $(t_{1,1},t_{1,2},\ldots,t_{k,k})\in\mathcal{B}$ and $j\in[k]$, define
 \begin{eqnarray*}
 m_j:=\sum_{i=1}^{k}t_{i,j}.
 \end{eqnarray*}
 For $\epsilon>0$, define a set $\mathcal{B}_{\epsilon}$ consisting of all
 $(t_{1,1},t_{1,2},\ldots,t_{k,k})\in \mathcal{B}$ satisfying all the following conditions:
 \begin{enumerate}
\item $\forall\ i\in[k],\ \max_{j\in[k]}t_{j,i}\geq m_i-n\epsilon$.
\item There exists a bijection $w:[k]\rightarrow [k]$, such that
\begin{eqnarray*}
 t_{w(i),i}=\max_{j\in[k]}t_{j,i},\qquad \forall i\in[k].
\end{eqnarray*}
\item $w$ preserves $\theta$, i.e., for any $x\in\Omega$, $i\in[p]$ and $a\in[k]$, we have
\begin{eqnarray*}
 \theta(x,i,a)=\theta(w\circ x,i,w(a)).
 \end{eqnarray*}
\end{enumerate}
\end{definition}

\noindent\textbf{Interpretation and role of Definitions \ref{df24}--\ref{dfb}.}
For two community assignment mappings $x,z\in\Omega$, the numbers
\[
t_{i,j}(x,z)=|x^{-1}(i)\cap z^{-1}(j)|,\qquad i,j\in[k],
\]
form a $k\times k$ overlap (contingency) table between the two partitions induced by
$x$ and $z$. Equivalently, $t_{i,j}(x,z)$ counts how many vertices are assigned label $i$
under $x$ while having label $j$ under $z$. The sets
$S_{i,j}(x,z)=x^{-1}(i)\cap z^{-1}(j)$ are precisely the cells of this table and they
partition $[n]$.

For every $x,z\in\Omega$, the overlap table $T(x,z):=(t_{i,j}(x,z))_{i,j\in[k]}$ satisfies
\[
\sum_{i=1}^{k}\sum_{j=1}^{k} t_{i,j}(x,z)=n,
\qquad
\sum_{i=1}^{k} t_{i,j}(x,z)=n_j(z),\qquad \forall j\in[k].
\]
Hence $T(x,z)\in\mathcal{B}$ for every $x,z\in\Omega$. When $z=y$ is the ground truth,
the column sums are exactly the true community sizes $n_j$.

The subset $\mathcal{B}_\epsilon$ is designed to capture the near-truth regime up to a
label relabeling, relative to the \emph{current second argument}. Concretely,
$T(x,z)\in\mathcal{B}_\epsilon$ means that for each community $i$ under $z$ (that is, each
column $i$), there exists a dominant label $w(i)$ under $x$ such that
\[
t_{w(i),i}(x,z)\ge n_i(z)-n\epsilon.
\]
Condition (2) in Definition~\ref{dfb} ensures that these dominant matches define a
bijection $w:[k]\to[k]$, providing a canonical alignment between the labels of $x$ and $z$.
Condition (3) further restricts $w$ to be $\theta$-preserving, because the likelihood is
invariant under such relabelings (cf.\ Lemma~\ref{lfe}). With the aligned assignment
$z^\star:=w\circ z\in C(z)$, we then have
\[
\bigl|\{v\in[n]:x(v)\neq z^\star(v)\}\bigr|
=
n-\sum_{i=1}^{k}t_{w(i),i}(x,z)
\le kn\epsilon.
\]
Thus $\mathcal{B}_\epsilon$ indeed corresponds to assignments that differ from each other
only on $O(n\epsilon)$ vertices after an appropriate $\theta$-preserving relabeling.

This formulation is needed not only for $T(x,y)$ with $y$ equal to the truth, but also for
$T(x,y_g)$ in the proof of Theorem~\ref{p215}, where the second argument is an intermediate
assignment whose community sizes need not coincide with those of the true assignment.

\medskip

\medskip
\noindent\textbf{Notation (single-vertex relabeling).}
For $y\in\Omega$, a vertex $v\in[n]$ and a label $a\in[k]$, denote by $y^{(v,a)}\in\Omega$
the assignment obtained from $y$ by changing only the label of $v$ to $a$, i.e.
\[
y^{(v,a)}(u)=
\begin{cases}
a, & u=v,\\
y(u), & u\neq v.
\end{cases}
\]

\begin{assumption}[One-step improvement in the aligned near-truth regime]\label{ap214}
Assume $\epsilon\in\bigl(0,\frac{2c}{3k}\bigr)$, $x\in\Omega_{2c/3}$ and the true assignment
$y\in\Omega_c$. Assume that there exists a positive deterministic sequence $\Delta_1=\Delta_1(n)>0$ such that the following holds. We suppress the dependence on $n$ in the notation.

Let $y_1\in\Omega_{2c/3}$ satisfy that 
\begin{align}
(t_{1,1}(x,y_1),t_{1,2}(x,y_1),\dots,t_{k,k}(x,y_1))\in \mathcal{B}_\epsilon,\label{txyle}
\end{align}
and that the canonical alignment is the identity in the sense that for every column $i\in[k]$,
\[
t_{i,i}(x,y_1)=\max_{j\in[k]} t_{j,i}(x,y_1),
\]
and that $y_1\notin C(x)$.
Pick any vertex $v\in[n]$ that is \emph{mis-labeled relative to $x$}, i.e. $x(v)=b$ and $y_1(v)=a$
for some $a\neq b$, and define $y_2:=y_1^{(v,b)}$, namely
\begin{equation}\label{y21}
y_2(u):=
\begin{cases}
b, & u=v,\\
y_1(u), & u\in[n]\setminus\{v\}.
\end{cases}
\end{equation}
Assume additionally that $y_2\in\Omega_{2c/3}$ (so the move keeps community sizes in the admissible range).
Then
\begin{equation}\label{g1}
L_\Sigma(x,y_1)-L_\Sigma(x,y_2)\ \ge\ \Delta_1\,(1+o(1)),
\end{equation}
where $o(1)\to 0$ as $n\to\infty$, uniformly over all admissible triples
$(x,y_1,v)$.
\end{assumption}

\begin{assumption}[Local summability of one-step ``backward'' moves]\label{ap29}
Assume $\epsilon\in\bigl(0,\frac{2c}{3k}\bigr)$, $x\in\Omega_{2c/3}$ and the true assignment
$y\in\Omega_c$.

Let $y_2\in\Omega_{2c/3}$. For each vertex $v\in[n]$ such that $y_2(v)=x(v)$ (a \emph{correctly labeled}
vertex relative to $x$), and for each label $a\in[k]\setminus\{x(v)\}$, define
$y_1^{(v,a)}:=y_2^{(v,a)}$, i.e.
\begin{equation}\label{y21b}
y_1^{(v,a)}(u):=
\begin{cases}
a, & u=v,\\
y_2(u), & u\in[n]\setminus\{v\}.
\end{cases}
\end{equation}
Consider only those pairs $(v,a)$ for which
\begin{align}
(t_{1,1}(x,y_1^{(v,a)}),t_{1,2}(x,y_1^{(v,a)}),\dots,t_{k,k}(x,y_1^{(v,a)}))\in \mathcal{B}_\epsilon,\label{txyl2}
\end{align} and diagonal-max condition
\[
t_{i,i}(x,y_1^{(v,a)})=\max_{j\in[k]} t_{j,i}(x,y_1^{(v,a)}),\qquad \forall i\in[k],
\]
hold, 
and $y_1^{(v,a)}\notin C(x)$. Then
\begin{equation}\label{g2}
\lim_{n\to\infty}\ 
\max_{x,y_2\in\Omega_{2c/3}}
\ \sum_{\substack{v\in[n]:\,y_2(v)=x(v)}}\ \sum_{a\in[k]\setminus\{x(v)\}}
\exp\!\left(
-\frac{L_\Sigma(x,y_1^{(v,a)})-L_\Sigma(x,y_2)}{8}
\right)\ =\ 0,
\end{equation}
where in the double sum we only include the pairs $(v,a)$ satisfying the above conditions.
\end{assumption}


\medskip
\begin{remark}[Statistical meaning and verifiability of Assumptions~\ref{ap214}--\ref{ap29}]
\label{rem:stat-meaning-local-assm}
For each assignment $x\in\Omega$, let $\mathbb P_x$ denote the law of the observation
$\overrightarrow{\mathbf{K}_x}=\overrightarrow{\mathbf{A}_x}+\overrightarrow{\mathbf{W}}$, where $\overrightarrow{\mathbf{W}}$ is centered Gaussian with
covariance $\Sigma$.

\smallskip
\noindent\textbf{(i) $L_\Sigma$ as an information distance.}
When $\Sigma$ is invertible, the Kullback--Leibler divergence between two Gaussian models
with the same covariance satisfies
\[
\mathrm{KL}(\mathbb P_x\,\|\,\mathbb P_z)
=\frac12\,(\overrightarrow{\mathbf{A}_x-\mathbf{A}_z)}^\top \Sigma^{-1}\,\overrightarrow{\mathbf{A}_x-\mathbf{A}_z}
=\frac12\,L_\Sigma(x,z).
\]
When $\Sigma$ is singular, the same identity holds for the induced Gaussian measures on
their common support (equivalently, after restricting to $\Im(\Sigma)$), while if
$\overrightarrow{\mathbf{A}_x-\mathbf{A}_z}\notin\Im(\Sigma)$ then the two measures are mutually singular and
the model is trivially identifiable along the noise-free directions. Consequently,
$L_\Sigma(x,z)$ should be viewed as (twice) the KL divergence / signal-to-noise separation
between the distributions indexed by $x$ and $z$.

\smallskip
\noindent\textbf{(ii) What $\Delta_1$ means.}
Assumption~\ref{ap214} postulates a \emph{uniform local information margin} in the aligned
near-truth regime: correcting one mis-labeled vertex decreases the ``energy''
$L_\Sigma(x,\cdot)$ by at least $\Delta_1(1+o(1))$. In view of (i), this is exactly a
uniform lower bound on the \emph{increment of KL divergence} (or, equivalently,
a local SNR margin) produced by a single-vertex correction step.

Importantly, $\Delta_1$ is \emph{deterministic given $(\theta,\Sigma)$ and $n$}: it depends
only on the mean map $x\mapsto A_x$ and the noise covariance $\Sigma$. One may make this
explicit by defining the deterministic ``worst-case one-step margin''
\[
\Delta_1^{\mathrm{wc}}(n)
:=\inf \Bigl\{ L_\Sigma(x,y_1)-L_\Sigma\bigl(x,y_1^{(v,x(v))}\bigr)\,:\;
(x,y_1,v)\ \text{satisfy the hypotheses of Assumption~\ref{ap214}}\Bigr\}.
\]
Then Assumption~\ref{ap214} is precisely the requirement that
$\Delta_1^{\mathrm{wc}}(n)\ge \Delta_1(1+o(1))$.

\smallskip
\noindent\textbf{(iii) What Assumption~\ref{ap29} means.}
Write the local ``backward'' energy increments
\[
\Delta_{v\to a}(x,y_2)
:=L_\Sigma\bigl(x,y_2^{(v,a)}\bigr)-L_\Sigma(x,y_2),
\qquad a\neq x(v).
\]
Assumption~\ref{ap29} requires that the combined exponential weight
$\sum_{v,a}\exp(-\Delta_{v\to a}(x,y_2)/8)$ is $o(1)$ uniformly over admissible $(x,y_2)$.
This is the standard \emph{energy--entropy} condition: the typical local KL/SNR costs
$\Delta_{v\to a}$ must dominate the $O(nk)$ number of available single-vertex perturbations.
In particular, a simple sufficient condition is
\[
\min_{v,a\neq x(v)} \Delta_{v\to a}(x,y_2)\ \gg\ \log(nk),
\]
which recovers the familiar ``SNR $\gtrsim \log n$'' scaling in classical exact recovery
thresholds.

\smallskip
\noindent\textbf{(iv) Sanity check: the i.i.d.\ GMM / $k$-means case.}
Consider the special case $\theta(x,i,a)\equiv \mu_{i,a}$ (no dependence on $x$) and
$\Sigma=\sigma^2 I_{pn}$. Then $A_x=[\mu_{x(1)},\dots,\mu_{x(n)}]$ and
\[
L_\Sigma(x,z)
=\frac{1}{\sigma^2}\sum_{u=1}^n \|\mu_{x(u)}-\mu_{z(u)}\|_2^2.
\]
Hence for a single correction at vertex $v$,
\[
L_\Sigma(x,y_1)-L_\Sigma\bigl(x,y_1^{(v,x(v))}\bigr)
=\frac{1}{\sigma^2}\|\mu_{x(v)}-\mu_{y_1(v)}\|_2^2,
\]
so $\Delta_1$ can be taken as
$\Delta_1 = \sigma^{-2}\min_{a\neq b}\|\mu_a-\mu_b\|_2^2$.
Therefore condition \eqref{ld2} in Theorem~\ref{p215} becomes the classical separation
threshold
\[
\min_{a\neq b}\|\mu_a-\mu_b\|_2^2 \gg \sigma^2\log(nk).
\]
More structured dependent models lead to analogous explicit expressions; see Section~\ref{sect:ex}
for the row-wise benchmark and the cross-vertex dependent examples.
\end{remark}

\medskip
\noindent\textbf{Interpretation of the local condition.}
Assumption \ref{ap214} formalizes the local one-step improvement mechanism in the
aligned near-truth regime $\mathcal{B}_\epsilon$. If $T(x,y_1)\in\mathcal{B}_\epsilon$ with
identity alignment and $y_1\notin C(x)$, then a vertex $v$ with $y_1(v)\neq x(v)$ is a
mismatch between $y_1$ and $x$. Assumption \ref{ap214} postulates a uniform one-step
margin: correcting such a vertex by changing its label from $y_1(v)$ to $x(v)$ decreases
the whitened separation $L_\Sigma(x,\cdot)$ by at least $\Delta_1(1+o(1))$.

Condition \eqref{ld2} is the corresponding energy--entropy requirement. Combined with
Assumption \ref{ap214}, it guarantees that the contribution of all near-truth alternatives
at Hamming distance $h$ is bounded by a summable geometric term of order
\[
\left(nk\exp\left(-\frac{\Delta_1(1-\eta)}{8}\right)\right)^h,
\]
which is exactly the mechanism used in the proof of Theorem~\ref{p215}.

Theorem \ref{p215} gives a sufficient condition for the exact recovery when the number of vertices in each community is unknown, and will be proved in Section \ref{sect:p215}.

\begin{theorem}\label{p215}
Assume $y\in \Omega_c$ is the true community assignment mapping. Suppose that
Assumptions \ref{ap24} and \ref{ap214} hold. Let $\epsilon\in(0,\frac{2c}{3k})$.
If
\begin{eqnarray}
\lim_{n\rightarrow\infty} n\log k-\frac{1}{8}\min_{x:(t_{1,1}(x,y),\ldots,t_{k,k}(x,y))\in \mathcal{B}\setminus\mathcal{B}_{\epsilon}}L_{\Sigma}(x,y)=-\infty,\label{ld1}
\end{eqnarray}
and there exists $\eta\in (0,1)$ independent of $n$ such that
\begin{eqnarray}
\lim_{n\rightarrow\infty}\log k+\log n-\frac{\Delta_1(1-\eta)}{8}=-\infty,\label{ld2}
\end{eqnarray}
then $\lim_{n\rightarrow\infty}\Pr(\hat{y}\in C(y))=1$.
\end{theorem}

\begin{remark}[Near-truth energy--entropy condition]\label{c211}
Condition \eqref{ld2} is the quantitative near-truth summability condition used in the proof
of Theorem~\ref{p215}. Under Assumption~\ref{ap214}, every single-vertex correction
improves $L_\Sigma$ by at least $\Delta_1(1+o(1))$, while the number of assignments at
distance $h$ from the aligned truth is at most $\binom{n}{h}(k-1)^h$.
\end{remark}

\begin{definition}\label{df16}Define the distance function $D_{\Omega}:\Omega\times\Omega\rightarrow [n]$ as follows
\begin{eqnarray*}
D_{\Omega}(x,y)=\sum_{i,j\in[k],i\neq j}t_{i,j}(x,y).
\end{eqnarray*}
for $x,y\in\Omega$.
\end{definition}

From Definition \ref{df16}, it is straightforward to check that
\begin{eqnarray*}
D_{\Omega}(x,y)=n-\sum_{i\in[k]}t_{i,i}(x,y)
\end{eqnarray*} 

\begin{assumption}\label{ap46}
Assume $x,y_m,y_h\in \Omega$ such that
\begin{enumerate}
\item $D_{\Omega}(y_m,y_h)=j$, where $j\geq 2$ is a positive integer; and
\item there exist distinct $u_1,\ldots,u_j\in[n]$, with the cyclic convention $u_0:=u_j$, such that
\begin{enumerate}
\item $y_m(v)=y_h(v)$, for all $v\in [n]\setminus \{u_1,\ldots,u_j\}$;
\item $y_m(u_i)\neq y_h(u_i)=x(u_i)=y_m(u_{i-1})$ for all $i\in[j]$;
\item $(t_{1,1}(x,y_m),t_{1,2}(x,y_m),\ldots,t_{k,k}(x,y_m))\in \mathcal{B}_{\epsilon}$ with $\epsilon\in\left(0,\frac{2c}{3k}\right)$ and $w(i)=i$.
\end{enumerate}
\end{enumerate}
Then
\begin{eqnarray}
L_{\Sigma}(x,y_m)- L_{\Sigma}(x,y_h)\geq j\Delta_2(1+o(1))\label{l46d}
\end{eqnarray}
for some positive deterministic sequence $\Delta_2=\Delta_2(n)>0$, uniformly over all such $x,y_m,y_h$. We suppress the dependence on $n$ in the notation.
\end{assumption}

Theorem \ref{m27} gives a sufficient condition for the exact recovery when the number of vertices in each community is known, and will be proved in Section \ref{sect:pm27}.

\begin{assumption}\label{ap27}
Assume $\epsilon\in(0,\frac{2c}{3k})$. Then for every $y\in \Omega_c$ and every
$x\in \Omega$ such that
\begin{eqnarray}
(t_{1,1}(x,y),t_{1,2}(x,y),\ldots,t_{k,k}(x,y))\in \mathcal{B}\setminus \mathcal{B}_{\epsilon},\label{tcd}
\end{eqnarray}
we have
\begin{eqnarray}
\left\|Q_r\overrightarrow{(\bA_x-\bA_y)}\right\|^2\geq T(n).\label{cap26}
\end{eqnarray}
Here $Q_r\overrightarrow{(\bA_x-\bA_y)}\in \RR^{r\times 1}$, and
$\left\|Q_r\overrightarrow{(\bA_x-\bA_y)}\right\|^2$ is the sum of squares of its $r$ components. This is a global-separation condition over all assignments outside the aligned near-truth regime; in particular, it is stated for all $x\in\Omega$ so that Lemma~\ref{l13} and Corollary~\ref{c212} apply directly to the global term in Theorem~\ref{p215}.
\end{assumption}

 \begin{remark}
 When $\Sigma$ is invertible, $Q_r=Q$, and (\ref{cap26}) becomes
 \begin{eqnarray*}
 \left\|\overrightarrow{(\bA_x-\bA_y)}\right\|^2\geq T(n),
 \end{eqnarray*}
 given that $Q$ is orthogonal.
 \end{remark}

\begin{theorem}\label{m27}Suppose that Assumptions \ref{ap27}, \ref{ap46} hold. Assume further that every \(\theta\)-preserving bijection \(w:[k]\to[k]\) preserves the prescribed community-size vector, namely
\[
  n_{w(i)}=n_i,\qquad i\in[k].
\]
If
\begin{eqnarray}
\lim_{n\rightarrow\infty} n\log k-\frac{T(n)}{8\lambda_1(n)}=-\infty,\label{ld3}
\end{eqnarray}
and there exists $\eta\in(0,1)$ independent of $n$ such that
\begin{eqnarray}
\lim_{n\rightarrow\infty}\log k+\log n-\frac{\Delta_2(1-\eta)}{8}=-\infty,\label{ld4}
\end{eqnarray}
then $\lim_{n\rightarrow\infty}\Pr(\check{y}\in C(y))=1$.
\end{theorem}

Theorem \ref{mm3} gives a necessary condition for the exact recovery (or equivalently, a sufficient condition for the failure of the exact recovery) when the number of vertices in each community is unknown and will be proved in Section \ref{sect:pm3}.

Recall that $y\in \Omega$ is the true community assignment mapping satisfying $|y^{-1}(i)|=n_i$, for all $i\in[k]$. 
 Let $a\in[n]$. 
Let $y^{(a)}\in\Omega$ be defined by 
\begin{eqnarray}
y^{(a)}(i)=\begin{cases} y(i)&\mathrm{if}\ i\in[n],\ \mathrm{and}\ i\neq a\\ y^{(a)}(a)&\mathrm{if}\ i=a.\end{cases}\label{dya}
\end{eqnarray}
such that
\begin{eqnarray*}
y(a)\neq y^{(a)}(a)\in[k].
\end{eqnarray*}

\begin{theorem}[Converse under asymptotically diagonal local statistics]
\label{mm3}
Assume \(\Sigma\) is invertible. Let \(H=H_n\subseteq[n]\) satisfy
\[
\log |H|=(1+o(1))\log n.
\]
For each \(a\in H\), let \(y^{(a)}\notin C(y)\) be a single-vertex perturbation such that
\[
L_\Sigma(y^{(a)},y)=\Delta_n(1+o(1))
\]
uniformly over \(a\in H\). Define
\[
\eta_a:=
\frac{
2(\operatorname{vec}(\mathbf A_{y^{(a)}}-\mathbf A_y))^\top
\Sigma^{-1}\operatorname{vec}(\mathbf W)
}{
L_\Sigma(y^{(a)},y)
}.
\]
Let \(\Phi_H\) be the covariance matrix of \(\{\eta_a:a\in H\}\). Assume
\[
\Phi_H=
\frac4{\Delta_n}(I+R_n),
\qquad
\|R_n\|_{\mathrm{op}}\to0.
\]
If
\[
\Delta_n\le(8-\delta)\log n
\]
for some fixed \(\delta>0\), then
\[
\Pr(\widehat y\in C(y))\to0.
\]
\end{theorem}

Theorem \ref{p31} gives a necessary condition for the exact recovery when the number of vertices in each community is known and will be proved in Section \ref{sect:pm4}.

\begin{remark}[Swap perturbations]
The most common size-preserving local alternatives are swaps. If $a,b\in[n]$ with
$y(a)\ne y(b)$, let $y^{(ab)}\in \Omega_{n_1,\ldots,n_k}$ be the assignment obtained
from $y$ by exchanging the labels of $a$ and $b$, namely
\begin{eqnarray}
y^{(ab)}(i)=\begin{cases}y(i)& \mathrm{if}\ i\in[n]\setminus\{a,b\}\\ y(b)&\mathrm{if}\ i=a\\ y(a)& \mathrm{if}\ i=b . \end{cases}\label{yab1}
\end{eqnarray}
Theorem~\ref{p31} is stated for a general family of size-preserving alternatives,
which includes such swaps as a special case.
\end{remark}

\begin{theorem}[Known-size converse under asymptotically diagonal local statistics]
\label{p31}
Assume \(\Sigma\) is invertible. Let \(\mathcal H_n\) be an index set. For each
\(\alpha\in\mathcal H_n\), let \(y^{(\alpha)}\in\Omega_{n_1,\ldots,n_k}\setminus C(y)\)
be a size-preserving alternative. Assume that
\[
\log |\mathcal H_n|=(1+o(1))\log n
\]
and
\[
L_\Sigma(y^{(\alpha)},y)=\Delta_n(1+o(1))
\]
uniformly over \(\alpha\in\mathcal H_n\). Define
\[
\eta_\alpha
:=
\frac{
2(\operatorname{vec}(\mathbf A_{y^{(\alpha)}}-\mathbf A_y))^\top
\Sigma^{-1}\operatorname{vec}(\mathbf W)
}{
L_\Sigma(y^{(\alpha)},y)
}.
\]
Let \(\Psi_{\mathcal H_n}\) be the covariance matrix of
\(\{\eta_\alpha:\alpha\in\mathcal H_n\}\). Assume
\[
\Psi_{\mathcal H_n}
=
\frac4{\Delta_n}(I+R_n),
\qquad
\|R_n\|_{\mathrm{op}}\to0.
\]
If
\[
\Delta_n\le(8-\delta)\log n
\]
for some fixed \(\delta>0\), then
\[
\Pr(\check y\in C(y))\to0.
\]
\end{theorem}

Given Theorems \ref{p215}, \ref{m27}, \ref{mm3}, \ref{p31}, it is natural to ask whether there exists a
sharp threshold for exact recovery, namely, whether the sufficient and necessary conditions match at a single deterministic scale. In the unknown-community-size setting, the matching mechanism is the following. If the one-step correction margin is \(\Delta_n(1+o(1))\), and if there exists a family \(H_n\) of single-vertex perturbations with
\[
\log |H_n|=(1+o(1))\log n,
\qquad
L_\Sigma(y^{(a)},y)=\Delta_n(1+o(1)),
\]
whose comparison-statistic covariance matrix satisfies
\[
\Phi_{H_n}=\frac4{\Delta_n}(I+R_n),
\qquad
\|R_n\|_{\mathrm{op}}\to0,
\]
then Theorem~\ref{mm3} gives failure below \((8-\delta)\log n\), while Theorem~\ref{p215} gives success above \((8+\delta)\log n\), once the corresponding global and local margin conditions hold. This is precisely the no-gap mechanism abstracted in Proposition~\ref{prop:compact-sharp-threshold}.

\subsection{Proof of the unknown-size sufficient theorem}

We now prove Theorem \ref{p215} and its corollary (Corollary \ref{c212}).

\subsection{Proof of Theorem \ref{p215}}

Recall that $y\in \Omega_{n_1,\ldots,n_k}$ is the true community assignment mapping.
Given the observation $\mathbf{K}_y$, define the feasible set
\[
\mathcal F(\mathbf{K}_y):=
\Bigl\{
x\in\Omega:\ (I-\Pi_\Sigma)\overrightarrow{(\mathbf{K}_y-\mathbf{A}_x)}=0
\Bigr\}.
\]
Since $\mathbf{K}_y=\mathbf{A}_y+\mathbf{W}$ and $\overrightarrow{\mathbf{W}}\in R_\Sigma$ almost surely, we have
$y\in \mathcal F(\mathbf{K}_y)$ almost surely.

By Lemma~\ref{le24}, the MLE $\hat y$ is the minimizer of $G_\Sigma(\cdot;\mathbf{K}_y)$ over
$\mathcal F(\mathbf{K}_y)$. For every $x\in \mathcal F(\mathbf{K}_y)$, we have
\[
G_\Sigma(x;\mathbf{K}_y)
=
\sum_{i,k\in[p];\,j,l\in[n]}
(\mathbf{K}_y)_{i,j}(\Sigma^\dagger)_{i,j;k,l}(\mathbf{K}_y)_{k,l}
+
\sum_{i,k\in[p];\,j,l\in[n]}
(\mathbf{A}_x)_{i,j}(\Sigma^\dagger)_{i,j;k,l}(\mathbf{A}_x)_{k,l}
\]
\[
\qquad\qquad
-2\sum_{i,k\in[p];\,j,l\in[n]}
(\mathbf{A}_x)_{i,j}(\Sigma^\dagger)_{i,j;k,l}(\mathbf{K}_y)_{k,l}.
\]
The first term is independent of $x$. Therefore, on the feasible set
$\mathcal F(\mathbf{K}_y)$, minimizing $G_\Sigma(x;\mathbf{K}_y)$ is equivalent to minimizing
\begin{align}
f(x):=
\sum_{i,k\in[p];\,j,l\in[n]}
(\mathbf{A}_x)_{i,j}(\Sigma^\dagger)_{i,j;k,l}(\mathbf{A}_x)_{k,l}
-
2\sum_{i,k\in[p];\,j,l\in[n]}
(\mathbf{A}_x)_{i,j}(\Sigma^\dagger)_{i,j;k,l}(\mathbf{K}_y)_{k,l}.
\label{31}
\end{align}

Consequently,
\begin{align}
\{\hat y\notin C(y)\}
\subseteq
\Bigl\{
\exists\, x\in \Omega\setminus C(y)\ \text{such that}\ f(x)\le f(y)
\Bigr\}.
\label{31a}
\end{align}
Indeed, if $\hat y\notin C(y)$, then there exists
$x\in \mathcal F(\mathbf{K}_y)\setminus C(y)$ such that
$G_\Sigma(x;\mathbf{K}_y)\le G_\Sigma(y;\mathbf{K}_y)$.

Then 
\begin{eqnarray}
&&f(x)-f(y)\label{fxmy}\\
&=&\sum_{i,k\in[p];j,l\in[n]}\left(\bA_x\right)_{i,j}(\Sigma^{\dagger})_{i,j;k,l}
\left(\bA_x\right)_{k,l}-\left(\bA_y\right)_{i,j}(\Sigma^{\dagger})_{i,j;k,l}
\left(\bA_y\right)_{k,l}\notag\\
&&-2\left[\left(\bA_x\right)_{i,j}-
\left(\bA_y\right)_{i,j}
\right](\Sigma^{\dagger})_{i,j;k,l}.
\left(\bA_y+\bW\right)_{k,l}\notag\\
&=&\sum_{i,k\in[p];j,l\in[n]}\left(\bA_x-\bA_y\right)_{i,j}(\Sigma^{\dagger})_{i,j;k,l}
\left(\bA_x-\bA_y\right)_{k,l}\notag\\
&&-2\left[\left(\bA_x\right)_{i,j}-
\left(\bA_y\right)_{i,j}
\right](\Sigma^{\dagger})_{i,j;k,l}.
\left(\bW\right)_{k,l}\notag
\end{eqnarray}

Then $f(x)-f(y)$ is a Gaussian random variable with mean value
\begin{eqnarray*}
\mathbb{E}\left(f(x)-f(y)\right)=L_{\Sigma}(x,y);
\end{eqnarray*}
Since $\overrightarrow{\mathbf{W}}$ is a (possibly degenerate) centered Gaussian vector in $\mathbb{R}^{pn}$
with covariance matrix $\Sigma$, the scalar linear functional
\[
\Big\langle \overrightarrow{\mathbf{A}_x-\mathbf{A}_y},\, \Sigma^\dagger \overrightarrow{\mathbf{W}}\Big\rangle
= (\overrightarrow{\mathbf{A}_x-\mathbf{A}_y})^\top \Sigma^\dagger \overrightarrow{\mathbf{W}}
\]
is Gaussian with
\[
\operatorname{Var}\!\left(\Big\langle \overrightarrow{\mathbf{A}_x-\mathbf{A}_y},\, \Sigma^\dagger \overrightarrow{\mathbf{W}}\Big\rangle\right)
= (\overrightarrow{\mathbf{A}_x-\mathbf{A}_y})^\top \Sigma^\dagger \Sigma \Sigma^\dagger (\overrightarrow{\mathbf{A}_x-\mathbf{A}_y}).
\]
By the Moore--Penrose identity $\Sigma^\dagger \Sigma \Sigma^\dagger=\Sigma^\dagger$, we obtain
\[
\operatorname{Var}(f(x)-f(y))
=4(\overrightarrow{\mathbf{A}_x-\mathbf{A}_y})^\top \Sigma^\dagger (\overrightarrow{\mathbf{A}_x-\mathbf{A}_y})
=4L_\Sigma(x,y).
\]

It is straightforward to check the following lemma.

\begin{lemma}[Invariance under $\theta$-preserving relabelings and a basic union bound]
\label{lfe}
Let $x,z\in \Omega$. If $x\in C(z)$, then $A_x=A_z$. In particular,
\[
K_x=K_z \qquad\text{and}\qquad f(x)=f(z).
\]
Consequently, if we define
\[
p(\hat y;\sigma):=\Pr(\hat y\in C(y)),
\]
then
\begin{equation}
\label{eq:unionbound-basic}
1-p(\hat y;\sigma)
\le
\sum_{C(x)\in \overline{\Omega},\, C(x)\neq C(y)}
\Pr\bigl(f(x)-f(y)\le 0\bigr)
\le
\sum_{C(x)\in \overline{\Omega},\, C(x)\neq C(y)}
\exp\!\left(-\frac{L_\Sigma(x,y)}{8}\right).
\end{equation}
\end{lemma}

\begin{proof}
Since $x\in C(z)$, Definition~\ref{dfeq}(2) gives
\[
\theta(x,i,x(j))=\theta(z,i,z(j)),\qquad \forall\, i\in[p],\ j\in[n].
\]
Hence $(\mathbf{A}_x)_{i,j}=(\mathbf{A}_z)_{i,j}$ for all $i,j$, i.e.\ $\mathbf{A}_x=\mathbf{A}_z$.
By the model definition $\mathbf{K}_x=\mathbf{A}_x+\mathbf{W}$, we immediately get $\mathbf{K}_x=\mathbf{K}_z$.
Also, by the definition of $f(\cdot)$ in (\ref{31}), $f(x)$ depends on $x$ only through
$\mathbf{A}_x$, so $f(x)=f(z)$.

Next, by (\ref{31a}),
\[
\{\hat y\notin C(y)\}
\subseteq
\{\exists\, C(x)\in \overline{\Omega},\ C(x)\neq C(y)\ \text{such that}\ f(x)\le f(y)\}.
\]
Since $f$ is constant on each equivalence class $C(x)$, the first inequality in
\eqref{eq:unionbound-basic} follows from the union bound.

For the second inequality, recall from (\ref{fxmy}) that for each fixed $x\in\Omega$,
the random variable $f(x)-f(y)$ is Gaussian with mean $L_\Sigma(x,y)$ and variance
$4L_\Sigma(x,y)$. Therefore,
\[
\Pr\bigl(f(x)-f(y)\le 0\bigr)
=
\Pr\!\left(
\xi\le -\frac{\sqrt{L_\Sigma(x,y)}}{2}
\right)
\le
\exp\!\left(-\frac{L_\Sigma(x,y)}{8}\right),
\]
where $\xi\sim N(0,1)$ and we used the standard Gaussian tail bound
$\Pr(\xi\le -t)\le e^{-t^2/2}$ for $t>0$.
This proves \eqref{eq:unionbound-basic}.
\end{proof}



\medskip
\noindent\textbf{Permutation invariance and counting by overlap tables.}
Let
\[
\mathfrak S_\theta
:=\Bigl\{\eta:[k]\to[k]\ \text{bijective}:\ 
\theta(x,i,a)=\theta(\eta\circ x,i,\eta(a))\ \ \forall\,x\in\Omega,\, i\in[p],\, a\in[k]\Bigr\}
\]
be the group of $\theta$-preserving permutations (cf.\ Definition~\ref{dfb}(3)).
Recall that our risk event is tested over \emph{equivalence classes} $C(x)$
(Definition~\ref{dfeq}), and Lemma~\ref{lfe} shows that $f(\cdot)$ is constant on each class $C(x)$.
Hence, all union bounds are taken over \emph{distinct equivalence classes} rather than
individual labelings; in particular, labelings that differ only by a $\theta$-preserving
permutation are not double-counted (up to a factor $|\mathfrak S_\theta|\le k!$).

For a fixed truth $y\in\Omega$ and an overlap table $t=(t_{i,j})\in \mathcal{B}$, define
\[
\Omega(t;y):=\{x\in\Omega: T(x,y)=t\},
\qquad
\overline\Omega(t;y):=\{C(x): x\in\Omega(t;y)\}.
\]

\begin{lemma}[Counting labelings with a prescribed overlap table]\label{le32}
Let $y\in\Omega$ with $n_j:=|y^{-1}(j)|$. For any $t=(t_{i,j})\in \mathcal{B}$ satisfying
\[
\sum_{i=1}^k t_{i,j}=n_j,\qquad \forall j\in[k],
\]
we have
\[
|\Omega(t;y)|
=\prod_{j=1}^k \binom{n_j}{t_{1,j},\dots,t_{k,j}}
=\prod_{j=1}^k \frac{n_j!}{\prod_{i=1}^k t_{i,j}!}.
\]
If the column-sum condition above fails for some $j$, then $\Omega(t;y)=\emptyset$.
Consequently,
\[
|\overline{\Omega}(t;y)|
\le |\Omega(t;y)|
\le k^n,
\]
and the difference between counting labelings and counting equivalence classes is at most
a factor $|\mathfrak S_\theta|\le k!$.
\end{lemma}

\begin{proof}
If $\sum_{i=1}^k t_{i,j}\neq n_j$ for some $j$, then no assignment $x\in\Omega$ can satisfy
$T(x,y)=t$, so $\Omega(t;y)=\emptyset$.

Assume now that $\sum_{i=1}^k t_{i,j}=n_j$ for all $j\in[k]$. Within each true community
$y^{-1}(j)$ of size $n_j$, to realize $T(x,y)=t$ one must choose exactly $t_{1,j}$ vertices
to receive label $1$ under $x$, then $t_{2,j}$ vertices to receive label $2$, etc. The number
of such assignments on $y^{-1}(j)$ equals the multinomial coefficient
$\binom{n_j}{t_{1,j},\dots,t_{k,j}}$. Multiplying over $j$ yields the formula.
The remaining inequalities are immediate.
\end{proof}

\noindent\textbf{Proof of Theorem \ref{p215}.}
By Lemma~\ref{lfe},
\[
1-p(\hat y;\sigma)
\le
\sum_{C(x)\in \ol{\Omega},\, C(x)\neq C(y)}
\exp\!\left(-\frac{L_\Sigma(x,y)}{8}\right).
\]

We first record that the near-truth regime $\mathcal{B}_\epsilon$ automatically forces
the candidate assignment to lie in $\Omega_{2c/3}$.

\medskip
\noindent\emph{Claim.}
If $x\in\Omega$ satisfies
\[
(t_{1,1}(x,y),\ldots,t_{k,k}(x,y))\in \mathcal{B}_{\epsilon},
\]
then $x\in \Omega_{2c/3}$.

\smallskip
\noindent\emph{Proof of claim.}
Let $w:[k]\to[k]$ be the bijection in Definition~\ref{dfb}.
Fix any $a\in[k]$, and let $j=w^{-1}(a)$.
Since the second argument is $y$, the $j$-th column sum equals $n_j$.
Hence
\[
|x^{-1}(a)|
=
\sum_{l=1}^k t_{a,l}(x,y)
\ge
t_{a,j}(x,y)
=
t_{w(j),j}(x,y)
\ge
n_j-n\epsilon.
\]
Since $y\in\Omega_c$, we have $n_j\ge cn$. Since $\epsilon<\frac{2c}{3k}\le \frac{c}{3}$,
\[
|x^{-1}(a)|\ge (c-\epsilon)n>\frac{2c}{3}n.
\]
As $a\in[k]$ was arbitrary, every community under $x$ has size at least $\frac{2c}{3}n$.
Thus $x\in\Omega_{2c/3}$.
This proves the claim.

\medskip

We now split the right-hand side into the globally separated regime and the aligned
near-truth regime:
\[
\sum_{C(x)\in \ol{\Omega},\, C(x)\neq C(y)}
\exp\!\left(-\frac{L_\Sigma(x,y)}{8}\right)
=
I_1+I_2,
\]
where
\[
I_1=
\sum_{\substack{
C(x)\in \ol{\Omega}:\\
(t_{1,1}(x,y),\ldots,t_{k,k}(x,y))\in \mathcal{B}\setminus \mathcal{B}_{\epsilon},\\
C(x)\neq C(y)
}}
\exp\!\left(-\frac{L_\Sigma(x,y)}{8}\right),
\]
and
\[
I_2=
\sum_{\substack{
C(x)\in \ol{\Omega}:\\
(t_{1,1}(x,y),\ldots,t_{k,k}(x,y))\in \mathcal{B}_{\epsilon},\\
C(x)\neq C(y)
}}
\exp\!\left(-\frac{L_\Sigma(x,y)}{8}\right).
\]

By Lemma~\ref{le32},
\begin{align*}
I_1
&\le
\sum_{\substack{
x\in \Omega:\\
(t_{1,1}(x,y),\ldots,t_{k,k}(x,y))\in \mathcal{B}\setminus \mathcal{B}_{\epsilon}
}}
\exp\!\left(-\frac{L_\Sigma(x,y)}{8}\right)\\
&\le
k^n
\exp\!\left(
-\frac{1}{8}
\min_{x:\,(t_{1,1}(x,y),\ldots,t_{k,k}(x,y))\in \mathcal{B}\setminus \mathcal{B}_{\epsilon}}
L_\Sigma(x,y)
\right).
\end{align*}
When \eqref{ld1} holds, we obtain
\begin{align}
\lim_{n\to\infty} I_1 = 0.
\label{pr1}
\end{align}

We next consider $I_2$. Since the summand only depends on the equivalence class $C(x)$,
for each class contributing to $I_2$, we may choose a representative, still denoted by $x$,
such that
\begin{eqnarray}
t_{i,i}(x,y)=\max_{j\in[k]} t_{j,i}(x,y),\qquad \forall i\in[k].
\label{diag-id}
\end{eqnarray}
Indeed, if $w$ is the bijection from Definition~\ref{dfb} associated with
$(t_{1,1}(x,y),\ldots,t_{k,k}(x,y))\in \mathcal{B}_\epsilon$, then $w^{-1}\circ x\in C(x)$ and,
by Lemma~\ref{lfe}, replacing $x$ by $w^{-1}\circ x$ does not change the summand.
Therefore,
\begin{eqnarray}
I_2
\le
\sum_{\substack{
x\in \Omega_{\frac{2c}{3}}:\\
(t_{1,1}(x,y),\ldots,t_{k,k}(x,y))\in \mathcal{B}_\epsilon,\\
t_{i,i}(x,y)=\max_{j\in[k]} t_{j,i}(x,y)\ \forall i\in[k],\\
x\notin C(y)
}}
\exp\!\left(-\frac{L_\Sigma(x,y)}{8}\right).
\label{i2rep}
\end{eqnarray}

Fix one such $x$. Let
\[
M(x):=\{v\in[n]:x(v)\neq y(v)\},
\qquad
h:=|M(x)|=D_\Omega(x,y).
\]
Write
\[
M(x)=\{u_1,\ldots,u_h\}
\]
in increasing order. Define $y_0,\ldots,y_h\in\Omega$ recursively by
\[
y_0:=y,
\]
and for $g\in[h]$,
\begin{eqnarray}
y_g(z):=
\begin{cases}
x(u_\ell), & \text{if } z=u_\ell \text{ for some }\ell\in[g],\\[1mm]
y(z), & \text{otherwise.}
\end{cases}
\label{yg-path}
\end{eqnarray}
Then $y_h=x$ and
\[
D_\Omega(x,y_g)=h-g,\qquad \forall g\in\{0,1,\ldots,h\}.
\]

For $g\in\{0,1,\ldots,h\}$ and $i\in[k]$, define
\[
o_i(g):=\bigl|\{u_\ell\in y^{-1}(i):\ell\le g\}\bigr|,
\qquad
r_i(g):=\bigl|\{u_\ell\in x^{-1}(i):\ell\le g\}\bigr|.
\]
Then
\[
n_i(y_g)=n_i-o_i(g)+r_i(g),
\qquad
t_{i,i}(x,y_g)=t_{i,i}(x,y)+r_i(g).
\]
Since only vertices in $y^{-1}(i)\setminus x^{-1}(i)$ leave the $i$-th community when going
from $y$ to $x$, we have
\[
o_i(g)\le n_i-t_{i,i}(x,y)\le n\epsilon.
\]
Hence
\[
n_i(y_g)\ge n_i-n\epsilon\ge cn-n\epsilon>\frac{2c}{3}n,
\qquad \forall i\in[k],
\]
because $\epsilon<\frac{2c}{3k}\le \frac{c}{3}$. Therefore
\[
y_g\in\Omega_{2c/3},\qquad \forall g\in\{0,1,\ldots,h\}.
\]
Moreover,
\[
t_{i,i}(x,y_g)
=
t_{i,i}(x,y)+r_i(g)
\ge
n_i-n\epsilon+r_i(g)
=
n_i(y_g)-n\epsilon+o_i(g)
\ge
n_i(y_g)-n\epsilon.
\]
Thus
\[
(t_{1,1}(x,y_g),\ldots,t_{k,k}(x,y_g))\in \mathcal{B}_\epsilon
\]
for all $g$. Since the off-diagonal mass in the $i$-th column is at most $n\epsilon$,
while
\[
t_{i,i}(x,y_g)\ge n_i(y_g)-n\epsilon > n\epsilon,
\]
we also have
\[
t_{i,i}(x,y_g)=\max_{j\in[k]} t_{j,i}(x,y_g),\qquad \forall i\in[k].
\]

Finally, for any $x'\in C(x)\setminus\{x\}$, there exists a nontrivial $\theta$-preserving
bijection $\eta:[k]\to[k]$ with $x'=\eta\circ x$. Choosing $a\in[k]$ such that $\eta(a)\neq a$,
all vertices in $x^{-1}(a)$ change label, and therefore
\[
D_\Omega(x,x')\ge |x^{-1}(a)|\ge \frac{2c}{3}n.
\]
On the other hand, since $(t_{1,1}(x,y),\ldots,t_{k,k}(x,y))\in\mathcal{B}_\epsilon$ with
identity alignment,
\[
h=D_\Omega(x,y)=n-\sum_{i=1}^k t_{i,i}(x,y)\le kn\epsilon<\frac{2c}{3}n.
\]
Hence
\[
D_\Omega(x,y_g)=h-g<\frac{2c}{3}n,\qquad \forall g<h,
\]
which implies
\[
y_g\notin C(x),\qquad \forall g\in\{0,1,\ldots,h-1\}.
\]

Therefore every step $y_{g-1}\to y_g$ is admissible in Assumption~\ref{ap214}. By the
uniformity in Assumption~\ref{ap214}, for the fixed $\eta\in(0,1)$ from \eqref{ld2}, all
sufficiently large $n$ satisfy
\[
L_\Sigma(x,y_{g-1})-L_\Sigma(x,y_g)\ge \Delta_1(1-\eta),
\qquad \forall g\in[h].
\]
Summing over $g=1,\ldots,h$ and using $y_h=x$, we obtain
\[
L_\Sigma(x,y)
=
\sum_{g=1}^h\Bigl(L_\Sigma(x,y_{g-1})-L_\Sigma(x,y_g)\Bigr)
\ge h\,\Delta_1(1-\eta).
\]

Let
\[
N_h:=\left|\left\{x\in\Omega:\ D_\Omega(x,y)=h\right\}\right|.
\]
A crude counting bound gives
\[
N_h\le \binom{n}{h}(k-1)^h\le (nk)^h.
\]
Therefore,
\begin{align*}
I_2
&\le
\sum_{h=1}^{\lfloor kn\epsilon\rfloor}
N_h \exp\!\left(-\frac{h\Delta_1(1-\eta)}{8}\right)\\
&\le
\sum_{h=1}^{\infty}
\left(
nk\exp\!\left(-\frac{\Delta_1(1-\eta)}{8}\right)
\right)^h.
\end{align*}
By \eqref{ld2},
\[
nk\exp\!\left(-\frac{\Delta_1(1-\eta)}{8}\right)\to 0,
\]
and hence
\begin{align}
\lim_{n\to\infty}I_2=0.
\label{pr2}
\end{align}

The theorem follows from \eqref{pr1} and \eqref{pr2}.
$\hfill\Box$

\subsection{Another corollary of Theorem \ref{p215}.}

Since $\Sigma$ is symmetric and positive semidefinite, there exists an orthogonal matrix $Q\in \RR^{pn\times pn}$, such that
 \begin{eqnarray*}
 \Sigma=Q^t \Lambda Q
 \end{eqnarray*}
 where
 \begin{eqnarray*}
 \Lambda=\mathrm{diag}[\lambda_1(n),\ldots,\lambda_{r}(n),0,\ldots,0];
 \end{eqnarray*}
 such that $r\in[pn]$ is the rank of $\Sigma$
 \begin{eqnarray*}
 \lambda_1(n)\geq \lambda_2(n)\geq\ldots \lambda_{r}(n)> 0
 \end{eqnarray*}
 are eigenvalues of $\Sigma$. Let 
 \begin{eqnarray*}
 \Lambda_r:&=&\mathrm{diag}[\lambda_1(n),\ldots,\lambda_{r}(n)];
 \end{eqnarray*}
 and assume that $Q_r\in \RR^{r\times pn}$, such that
 \begin{eqnarray*}
 Q=\left(\begin{array}{c}Q_r\\\overline{Q}_r\end{array}\right).
 \end{eqnarray*}
 Then by Proposition \ref{pa3}, we obtain
 \begin{eqnarray}
 \Sigma^{\dagger}=Q_r^t\Lambda_r^{-1} Q_r.\label{sgi}
 \end{eqnarray}
 
 For $x\in \Omega$, define $\Upsilon_x,\Gamma_x\in \RR^{pn\times 1}$ by
  \begin{eqnarray}
[\Upsilon_x]_{n(i-1)+j}:&=&2(\bA_x-\bA_y)_{i,j};\qquad\forall i\in[p],\ j\in [n].\label{dux}
 \end{eqnarray}
 \begin{eqnarray}
[\Gamma_x]_{n(i-1)+j}:&=&\frac{2(\bA_x-\bA_y)_{i,j}}{L_{\Sigma}(x,y)};\qquad\forall i\in[p],\ j\in [n].\label{dgx}
 \end{eqnarray}

\begin{lemma}\label{l13} Assume that $\theta,\Sigma$ satisfies Assumptions \ref{ap27}.  Then for all the $x,y\in\Omega$ such that (\ref{tcd}) holds, we have
 \begin{eqnarray*}
 L_{\Sigma}(x,y)\geq \frac{T(n)}{\lambda_1(n)}.
 \end{eqnarray*}
 where $\lambda_1(n)>0$ is the maximal eigenvalue of $\Sigma$.
 \end{lemma}

 \begin{proof}
Recall that $\Sigma$ is symmetric positive semidefinite, so there exists an orthogonal
matrix $Q\in \mathbb{R}^{pn\times pn}$ such that
\[
\Sigma = Q^\top \Lambda Q,\qquad 
\Lambda = \mathrm{diag}\bigl(\lambda_1(n),\dots,\lambda_r(n),0,\dots,0\bigr),
\]
where $\lambda_1(n)\ge \cdots \ge \lambda_r(n)>0$ and $r=\mathrm{rank}(\Sigma)$.
Let $\Lambda_r:=\mathrm{diag}\bigl(\lambda_1(n),\dots,\lambda_r(n)\bigr)$ and let
$Q_r\in \mathbb{R}^{r\times pn}$ be the submatrix formed by the first $r$ rows of $Q$.
Then $\Sigma^\dagger = Q_r^\top \Lambda_r^{-1} Q_r$.

For $x\in\Omega$, recall $\Upsilon_x$ defined in \eqref{dux}. Then
\[
\Upsilon_x = 2\overrightarrow{(\bA_x-\bA_y)}.
\]
Set
\[
u := Q_r \Upsilon_x \in \mathbb{R}^r.
\]
Then
\[
L_\Sigma(x,y)
= \frac{1}{4}\,\Upsilon_x^\top \Sigma^{\dagger} \Upsilon_x
= \frac{1}{4}\,u^\top \Lambda_r^{-1} u.
\]
Since $\Lambda_r^{-1}$ is positive definite with eigenvalues
$\{1/\lambda_i(n)\}_{i=1}^r$, we have
\[
u^\top \Lambda_r^{-1} u \ge \frac{1}{\lambda_1(n)}\,u^\top u
= \frac{1}{\lambda_1(n)}\,\|Q_r\Upsilon_x\|_2^2.
\]
Therefore
\[
L_\Sigma(x,y)
\ge \frac{1}{4\lambda_1(n)}\,\|Q_r\Upsilon_x\|_2^2
= \frac{1}{\lambda_1(n)}\left\|Q_r\overrightarrow{(\bA_x-\bA_y)}\right\|^2.
\]
Under Assumption~\ref{ap27},
\[
\left\|Q_r\overrightarrow{(\bA_x-\bA_y)}\right\|^2\ge T(n).
\]
Hence
\[
L_\Sigma(x,y)\ge \frac{T(n)}{\lambda_1(n)}.
\]
\end{proof}

 \begin{corollary}\label{c212}
Assume $y\in \Omega_c$ is the true community assignment mapping. Let $\lambda_1(n)>0$
be the maximal eigenvalue of $\Sigma$. Suppose that Assumptions \ref{ap24}, \ref{ap27} and \ref{ap214} hold. Let $\varepsilon\in(0,\frac{2c}{3k})$. If
\begin{align}
\lim_{n\to\infty}
\left(
n\log k-\frac{T(n)}{8\lambda_1(n)}
\right)
=-\infty,
\label{pr17}
\end{align}
and (\ref{ld2}) hold,
then
\[
\lim_{n\to\infty}\Pr(\hat y\in C(y))=1.
\]
\end{corollary}

\begin{proof}
By Lemma~\ref{l13}, condition \eqref{pr17} implies condition~(\ref{ld1}). Therefore the conclusion
follows from Theorem~\ref{p215}.
\end{proof}

\section{Exact Recovery with Known Community Sizes}\label{sect:pm27}

In this section, we prove Theorem \ref{m27}.

For each $x \in \Omega_{n_1,\ldots,n_k}$, let
\[
C^*(x):=C(x)\cap \Omega_{n_1,\ldots,n_k};
\]
i.e.\ $C^*(x)$ consists of all the community assignment mappings in
$\Omega_{n_1,\ldots,n_k}$ that are equivalent to $x$ in the sense of Definition~\ref{dfeq}.
Let
\[
\overline{\Omega}_{n_1,\ldots,n_k}
:=
\{C^*(x): x\in \Omega_{n_1,\ldots,n_k}\};
\]
i.e.\ $\overline{\Omega}_{n_1,\ldots,n_k}$ consists of all equivalence classes in
$\Omega_{n_1,\ldots,n_k}$.

Given the observation $\mathbf{K}_y$, define the feasible set in the known-size case by
\[
\mathcal F_{n_1,\ldots,n_k}(\mathbf{K}_y)
:=
\Bigl\{
x\in\Omega_{n_1,\ldots,n_k}:\ (I-\Pi_\Sigma)\overrightarrow{(\mathbf{K}_y-\mathbf{A}_x)}=0
\Bigr\}.
\]
Since $\mathbf{K}_y=\mathbf{A}_y+\mathbf{W}$ and $\overrightarrow{\mathbf{W}}\in R_\Sigma$ almost surely, we have
$y\in \mathcal F_{n_1,\ldots,n_k}(\mathbf{K}_y)$ almost surely.

By Lemma~\ref{le24}, the MLE $\check y$ minimizes $G_\Sigma(\cdot;\mathbf{K}_y)$ over
$\mathcal F_{n_1,\ldots,n_k}(\mathbf{K}_y)$. Hence, if $\check y\notin C^*(y)$, then there exists
$x\in \mathcal F_{n_1,\ldots,n_k}(\mathbf{K}_y)\setminus C^*(y)$ such that
\[
G_\Sigma(x;\mathbf{K}_y)\le G_\Sigma(y;\mathbf{K}_y).
\]
Since both $x$ and $y$ are feasible, this implies $f(x)\le f(y)$.

Moreover, by Lemma~\ref{lfe}, if $x,z\in \Omega_{n_1,\ldots,n_k}$ and $x\in C^*(z)$,
then $\mathbf{A}_x=\mathbf{A}_z$ and hence $f(x)=f(z)$. Therefore,
\begin{equation}
1-p(\check y;\sigma)
\le
\sum_{C^*(x)\in \overline{\Omega}_{n_1,\ldots,n_k}\setminus\{C^*(y)\}}
\Pr\bigl(f(x)-f(y)\le 0\bigr)
\le
\sum_{C^*(x)\in \overline{\Omega}_{n_1,\ldots,n_k}\setminus\{C^*(y)\}}
\exp\!\left(-\frac{L_\Sigma(x,y)}{8}\right).
\label{wmc}
\end{equation}

\noindent{\textbf{Proof of Theorem \ref{m27}}.} Let
\[
J:=
\sum_{C^*(x)\in \overline{\Omega}_{n_1,\ldots,n_k}\setminus\{C^*(y)\}}
\exp\!\left(-\frac{L_\Sigma(x,y)}{8}\right).
\]
By \eqref{wmc}, it suffices to show that
\[
\lim_{n\to\infty} J=0.
\]

Let $0<\epsilon<\frac{2c}{3k}$.

Note that
\[
J \le J_1 + J_2,
\]
where
\[
J_1=
\sum_{\substack{
C^*(x)\in \overline{\Omega}_{n_1,\ldots,n_k}:\\
(t_{1,1}(x,y),\ldots,t_{k,k}(x,y))\in \mathcal{B}\setminus \mathcal{B}_\epsilon,\\
C^*(x)\neq C^*(y)
}}
\exp\!\left(-\frac{L_\Sigma(x,y)}{8}\right),
\]
and
\[
J_2=
\sum_{\substack{
C^*(x)\in \overline{\Omega}_{n_1,\ldots,n_k}:\\
(t_{1,1}(x,y),\ldots,t_{k,k}(x,y))\in \mathcal{B}_\epsilon,\\
C^*(x)\neq C^*(y)
}}
\exp\!\left(-\frac{L_\Sigma(x,y)}{8}\right).
\tag{4.2}
\]

Under Assumption \ref{ap27}, by Lemma \ref{l13} we have
\[
0\le J_1 \le k^n \exp\!\left(-\frac{T(n)}{8\lambda_1(n)}\right).
\]
By~(\ref{ld3}), we have
\begin{align}
\lim_{n\to\infty}J_1=0.\label{glz}
\tag{4.3}
\end{align}

We now prove that $J_2\to 0$.

For each class $C^*(x)$ contributing to $J_2$, let $w$ be the bijection in Definition~\ref{dfb}
associated with $(t_{1,1}(x,y),\ldots,t_{k,k}(x,y))\in \mathcal{B}_\epsilon$.
Since the bijection $w$ in Definition~\ref{dfb} is $\theta$-preserving, the size-preserving
assumption in Theorem~\ref{m27} gives
\[
  n_{w(i)}=n_i,\qquad i\in[k].
\]
Hence $w^{-1}\circ x\in \Omega_{n_1,\ldots,n_k}\cap C(x)=C^*(x)$.
Replacing $x$ by $w^{-1}\circ x$, we may assume that
\[
t_{i,i}(x,y)=\max_{j\in[k]} t_{j,i}(x,y),\qquad \forall i\in[k].
\]
Let $\mathcal{R}_\epsilon$ be a set containing one such representative for each equivalence class
counted in $J_2$. Then
\[
J_2=\sum_{x\in \mathcal{R}_\epsilon}\exp\!\left(-\frac{L_\Sigma(x,y)}{8}\right).
\]

Fix $x\in \mathcal{R}_\epsilon$. Let
\[
M(x):=\{v\in[n]:x(v)\neq y(v)\},
\qquad
h:=|M(x)|=D_\Omega(x,y).
\]
Consider the directed multigraph $G_x$ on the vertex set $[k]$ obtained by putting one
directed edge $y(v)\to x(v)$ for each $v\in M(x)$.
Since $x,y\in \Omega_{n_1,\ldots,n_k}$, for every $a\in[k]$ the out-degree and in-degree of $a$
in $G_x$ are equal. Therefore the edge set of $G_x$ can be decomposed into edge-disjoint
directed cycles. Equivalently, there exist integers $s\ge 1$ and $j_1,\ldots,j_s\ge 2$,
together with pairwise distinct vertices
\[
u_{r,1},\ldots,u_{r,j_r}\in M(x),\qquad r\in[s],
\]
such that
\[
\sum_{r=1}^s j_r=h,
\]
and, with the cyclic convention $u_{r,0}:=u_{r,j_r}$,
\[
x(u_{r,i})=y(u_{r,i-1}),\qquad \forall r\in[s],\ \forall i\in[j_r].
\]

We now construct assignments $z_0,\ldots,z_s\in \Omega_{n_1,\ldots,n_k}$ recursively.
Let
\[
z_0:=y,
\]
and for each $r\in[s]$, define $z_r$ by
\[
z_r(v):=
\begin{cases}
x(u_{r,i}), & \text{if } v=u_{r,i}\text{ for some }i\in[j_r],\\[1mm]
z_{r-1}(v), & \text{otherwise.}
\end{cases}
\]
By construction, each step $z_{r-1}\mapsto z_r$ performs a cyclic permutation of labels on
$\{u_{r,1},\ldots,u_{r,j_r}\}$, and hence preserves the community sizes:
\[
z_r\in \Omega_{n_1,\ldots,n_k},\qquad \forall r\in[s].
\]
Moreover, the cycles are disjoint and each corrected vertex is assigned its label under $x$,
so
\[
z_s=x.
\]

We claim that for every $r\in[s]$,
\[
(t_{1,1}(x,z_{r-1}),t_{1,2}(x,z_{r-1}),\ldots,t_{k,k}(x,z_{r-1}))\in \mathcal{B}_\epsilon
\quad\text{with }w(i)=i.
\]
Indeed, each step from $y$ to $z_{r-1}$ only changes some vertices from an incorrect label
(relative to $x$) to the correct one, so for every $i\in[k]$,
\[
t_{i,i}(x,z_{r-1})\ge t_{i,i}(x,y)\ge n_i-n\epsilon.
\]
Hence the diagonal-max condition with identity alignment remains valid throughout the path.

Also, by construction, for every $r\in[s]$ and every $i\in[j_r]$,
\[
z_{r-1}(u_{r,i})\neq z_r(u_{r,i})=x(u_{r,i})=z_{r-1}(u_{r,i-1}),
\]
and
\[
z_{r-1}(v)=z_r(v),\qquad \forall v\in[n]\setminus\{u_{r,1},\ldots,u_{r,j_r}\}.
\]
Therefore the triple $(x,z_{r-1},z_r)$ satisfies Assumption~\ref{ap46} with $j=j_r$.
Fix $\eta\in(0,1)$ satisfying \eqref{ld4}. By Assumption~\ref{ap46}, for all sufficiently large $n$,
uniformly over $r\in[s]$,
\[
L_\Sigma(x,z_{r-1})-L_\Sigma(x,z_r)\ge j_r\Delta_2(1-\eta).
\]
Summing over $r=1,\ldots,s$ and using $z_0=y$ and $z_s=x$, we obtain
\[
L_\Sigma(x,y)
=
\sum_{r=1}^s \Bigl(L_\Sigma(x,z_{r-1})-L_\Sigma(x,z_r)\Bigr)
\ge
\sum_{r=1}^s j_r\Delta_2(1-\eta)
=
h\Delta_2(1-\eta).
\]

Let
\[
N_h:=\left|\left\{x\in \Omega_{n_1,\ldots,n_k}:D_\Omega(x,y)=h\right\}\right|.
\]
A crude counting bound gives
\[
N_h\le \binom{n}{h}(k-1)^h\le (nk)^h.
\]
Therefore,
\[
J_2
\le
\sum_{h=1}^n N_h \exp\!\left(-\frac{h\Delta_2(1-\eta)}{8}\right)
\le
\sum_{h=1}^{\infty}
\left(
nk\exp\!\left(-\frac{\Delta_2(1-\eta)}{8}\right)
\right)^h.
\]
Set
\[
\beta_n:=nk\exp\!\left(-\frac{\Delta_2(1-\eta)}{8}\right).
\]
By~(\ref{ld4}), we have $\beta_n\to 0$. Hence, for all sufficiently large $n$,
\[
0\le J_2\le \sum_{h=1}^{\infty}\beta_n^h=\frac{\beta_n}{1-\beta_n}\to 0.
\]
Thus
\[
\lim_{n\to\infty}J_2=0.
\]
Combined with \eqref{glz}, this yields
\[
\lim_{n\to\infty}J=0.
\]
The theorem now follows from \eqref{wmc}.
$\hfill\Box$

\section{No Exact Recovery with Unknown Community Sizes}
\label{sect:pm3}
In this section, we prove Theorem \ref{mm3}.

\subsection{A Gaussian maximum bound under asymptotically diagonal covariance}\label{sec:gauss-max}

Throughout this subsection, we let $X_1,\ldots,X_N$ be $N$ Gaussian random variables with 0 mean. Let
\begin{eqnarray*}
M_N:=\max_{i\in[N]} X_i.
\end{eqnarray*}

\begin{lemma}(Theorem A.1 in \cite{ZL20})Let $\epsilon\in (0,1)$. Then 
\begin{eqnarray*}
\mathrm{Pr}\left(M_N>(1+\epsilon)\sqrt{2\max_{i\in[N]}\mathrm{Var}(X_i)\log N}\right)\leq N^{-\epsilon}
\end{eqnarray*}
\end{lemma}

\begin{lemma}[Gaussian maximum under asymptotically diagonal covariance]
\label{lem:gaussian-max-asymp-diagonal}
Let \(X_1,\ldots,X_N\) be centered Gaussian random variables with covariance matrix
\[
\Phi_N=\sigma_N^2(I_N+R_N),
\]
where \(R_N\) is symmetric and
\[
\sigma_N>0,
\qquad
\|R_N\|_{\mathrm{op}}\to0.
\]
Then, for every fixed \(\varepsilon\in(0,1)\),
\[
\Pr\left(
\max_{1\le i\le N}X_i
\ge
(1-\varepsilon)\sigma_N\sqrt{2\log N}
\right)
\to1.
\]
\end{lemma}

\begin{proof}
Let
\[
\delta_N:=\|R_N\|_{\mathrm{op}}.
\]
Since \(\delta_N\to0\), for all sufficiently large \(N\), the matrix \(I_N+R_N\) is
positive semidefinite and its eigenvalues lie in \([1-\delta_N,1+\delta_N]\).

Let \(Z=(Z_1,\ldots,Z_N)^\top\sim N(0,I_N)\). We may realize \(X=(X_1,\ldots,X_N)^\top\)
as
\[
X=\sigma_N (I_N+R_N)^{1/2}Z.
\]
Write
\[
Y:=\sigma_N Z,
\qquad
E:=X-Y
=
\sigma_N\left\{(I_N+R_N)^{1/2}-I_N\right\}Z.
\]
Set
\[
A_N:=(I_N+R_N)^{1/2}-I_N.
\]
By spectral calculus,
\[
\|A_N\|_{\mathrm{op}}
=
\max_{\lambda\in\mathrm{Spec}(R_N)}
\left|\sqrt{1+\lambda}-1\right|
\le
C\delta_N
\]
for an absolute constant \(C\) and all sufficiently large \(N\). Hence, for each coordinate,
\[
\operatorname{Var}(E_i)
=
\sigma_N^2\|(A_N)_{i,\cdot}\|_2^2
\le
\sigma_N^2\|A_N\|_{\mathrm{op}}^2
\le
C^2\sigma_N^2\delta_N^2.
\]
Therefore, for any fixed \(\eta>0\),
\[
\Pr\left(
\max_{1\le i\le N}|E_i|
>
\eta\sigma_N\sqrt{2\log N}
\right)
\le
2N
\exp\left(
-\frac{\eta^2\,2\log N}{2C^2\delta_N^2}
\right).
\]
Since \(\delta_N\to0\), the right-hand side tends to \(0\). Thus
\[
\max_{1\le i\le N}|E_i|
=
o_{\mathbb P}\bigl(\sigma_N\sqrt{\log N}\bigr).
\]

It remains to recall the standard lower bound for the maximum of independent standard
Gaussians. Since \(Z_1,\ldots,Z_N\) are i.i.d. \(N(0,1)\), for every fixed
\(\eta\in(0,1)\),
\[
\Pr\left(
\max_{1\le i\le N}Z_i
\ge
(1-\eta)\sqrt{2\log N}
\right)
\to1.
\]
Indeed, by Mills' bound,
\[
N\Pr\left(Z_1\ge(1-\eta)\sqrt{2\log N}\right)\to\infty,
\]
and hence the probability that none of the \(Z_i\)'s exceeds this level tends to zero.

Now take \(\eta=\varepsilon/2\). With probability tending to one,
\[
\max_{1\le i\le N}Y_i
=
\sigma_N\max_{1\le i\le N}Z_i
\ge
(1-\varepsilon/2)\sigma_N\sqrt{2\log N},
\]
and also
\[
\max_{1\le i\le N}|E_i|
\le
(\varepsilon/2)\sigma_N\sqrt{2\log N}.
\]
On the intersection of these two events,
\[
\max_{1\le i\le N}X_i
\ge
\max_{1\le i\le N}Y_i-\max_{1\le i\le N}|E_i|
\ge
(1-\varepsilon)\sigma_N\sqrt{2\log N}.
\]
This proves the lemma.
\end{proof}



\subsection{Proof of Theorem \ref{mm3}}

\begin{proof}
For each \(a\in H\),
\[
G_\Sigma(y^{(a)};\mathbf K_y)-G_\Sigma(y;\mathbf K_y)
=
L_\Sigma(y^{(a)},y)(1-\eta_a).
\]
Thus exact recovery fails whenever
\[
\max_{a\in H}\eta_a\ge1.
\]
By Lemma~\ref{lem:gaussian-max-asymp-diagonal},
\[
\max_{a\in H}\eta_a
\ge
(1-o_{\mathbb P}(1))
\sqrt{\frac{8\log |H|}{\Delta_n}}.
\]
Since \(\log |H|=(1+o(1))\log n\) and
\[
\Delta_n\le(8-\delta)\log n,
\]
the right-hand side is larger than \(1\) with probability tending to one. Hence
\[
\Pr(\widehat y\in C(y))\to0.
\]
\end{proof}

\section{No Exact Recovery with Known Community Sizes}
\label{sect:pm4}

In this section, we prove Theorem \ref{p31}.

\subsection{Proof of Theorem~\ref{p31}}

\begin{proof}
Since \(\Sigma\) is invertible, the support constraint is vacuous. For each \(\alpha\in\mathcal H_n\), the assignment \(y^{(\alpha)}\) belongs to the size-constrained space \(\Omega_{n_1,\ldots,n_k}\) and is not in \(C(y)\). The likelihood comparison gives
\[
G_\Sigma(y^{(\alpha)};\mathbf K_y)-G_\Sigma(y;\mathbf K_y)
=
L_\Sigma(y^{(\alpha)},y)(1-\eta_\alpha).
\]
Thus size-constrained exact recovery fails whenever
\[
\max_{\alpha\in\mathcal H_n}\eta_\alpha\ge1.
\]
By Lemma~\ref{lem:gaussian-max-asymp-diagonal}, applied to the Gaussian vector
\(\{\eta_\alpha:\alpha\in\mathcal H_n\}\) with
\(\sigma_N^2=4/\Delta_n\),
\[
\max_{\alpha\in\mathcal H_n}\eta_\alpha
\ge
(1-o_{\mathbb P}(1))
\sqrt{\frac{8\log |\mathcal H_n|}{\Delta_n}}.
\]
Since \(\log |\mathcal H_n|=(1+o(1))\log n\) and
\(\Delta_n\le(8-\delta)\log n\), the right-hand side is larger than \(1\) with probability tending to one. Therefore
\[
\Pr(\check y\in C(y))\to0.
\]
\end{proof}

\section{Examples Beyond Row-Wise Whitening}\label{sect:ex}

\subsection{A row-independent whitening benchmark}\label{st71}

We first revisit the row-independent product-covariance case. This example belongs to
the benchmark of Proposition~\ref{prop:row-wise-whitening}; it is included as a sanity
check showing that the general likelihood criterion recovers the classical whitened
threshold. The main cross-vertex dependent examples are given in the following two
subsections.

We now consider a covariance matrix that is non-diagonal but positive definite, so that
the sufficient condition in Section~\ref{sect:p215} and the converse in Section~\ref{sect:pm3}
match sharply.

Let $k=2$ and $p=2$. Fix $\rho\in(-1,1)\setminus\{0\}$, and for each $n$ let
\[
u_n := (\mu_n,-\mu_n)^t \in \mathbb R^2,
\qquad
C_\rho :=
\begin{pmatrix}
1 & \rho\\
\rho & 1
\end{pmatrix}.
\]
Define the mean map $\theta$ by
\[
\theta(x,1,1)=\mu_n,\qquad \theta(x,2,1)=-\mu_n,
\]
\[
\theta(x,1,2)=-\mu_n,\qquad \theta(x,2,2)=\mu_n,
\]
for all $x\in \Omega$. Equivalently, for each $x\in\Omega$ and $j\in[n]$,
\[
(\mathbf{A}_x)_{\cdot,j}=
\begin{cases}
u_n, & x(j)=1,\\[1mm]
-u_n, & x(j)=2.
\end{cases}
\]
Let $W_{\cdot,1},\dots,W_{\cdot,n}$ be i.i.d.\ Gaussian vectors in $\mathbb R^2$ with
distribution $N(0,C_\rho)$, independent across $j$, and observe
\[
\mathbf{K}_y=\mathbf{A}_y+\mathbf{W}.
\]
Then
\[
\Sigma := \mathrm{Cov}(\overrightarrow{\mathbf{W}}) = I_n\otimes C_\rho.
\]
Since $|\rho|<1$, the matrix $C_\rho$ is positive definite, hence $\Sigma$ is positive
definite and therefore invertible. Since $\rho\neq 0$, the matrix $\Sigma$ is not diagonal.

Moreover, in this example there is no nontrivial $\theta$-preserving permutation of
$\{1,2\}$, because $u_n\neq -u_n$. Hence
\[
C(y)=\{y\}.
\]

\begin{proposition}\label{p71}
In the above model,
\[
\Sigma^{-1}=I_n\otimes C_\rho^{-1},
\qquad
C_\rho^{-1}=
\frac{1}{1-\rho^2}
\begin{pmatrix}
1 & -\rho\\
-\rho & 1
\end{pmatrix}.
\]
For every $x\in\Omega$,
\[
L_\Sigma(x,y)=\Delta_n\, D_\Omega(x,y),
\]
where
\[
\Delta_n:=4u_n^t C_\rho^{-1}u_n
=
\frac{8\mu_n^2}{1-\rho}.
\]
\end{proposition}

\begin{proof}
For each vertex $j\in[n]$, the $j$-th column of $A_x-A_y$ is equal to $0$ if $x(j)=y(j)$,
and is equal to $\pm 2u_n$ if $x(j)\neq y(j)$. Since
\[
\Sigma^{-1}=I_n\otimes C_\rho^{-1},
\]
there are no cross-column terms in the quadratic form defining $L_\Sigma(x,y)$. Hence
\[
L_\Sigma(x,y)
=
\sum_{j:x(j)\neq y(j)}
(2u_n)^t C_\rho^{-1}(2u_n)
=
4u_n^t C_\rho^{-1}u_n\, D_\Omega(x,y).
\]
It remains to compute $u_n^t C_\rho^{-1}u_n$. Since
\[
C_\rho^{-1}=
\frac{1}{1-\rho^2}
\begin{pmatrix}
1 & -\rho\\
-\rho & 1
\end{pmatrix},
\]
we obtain
\[
u_n^t C_\rho^{-1}u_n
=
\frac{1}{1-\rho^2}
(\mu_n,-\mu_n)
\begin{pmatrix}
1 & -\rho\\
-\rho & 1
\end{pmatrix}
\binom{\mu_n}{-\mu_n}
=
\frac{2\mu_n^2}{1-\rho}.
\]
Therefore
\[
\Delta_n=4u_n^t C_\rho^{-1}u_n=\frac{8\mu_n^2}{1-\rho}.
\]
\end{proof}

Write
\[
T(x,y):=(t_{1,1}(x,y),t_{1,2}(x,y),t_{2,1}(x,y),t_{2,2}(x,y)).
\]

\begin{proposition}\label{p72}
Assume $y\in \Omega_c$ for some fixed $c\in(0,1/2)$, and let
\[
\varepsilon\in \left(0,\frac{2c}{3k}\right)=\left(0,\frac{c}{3}\right).
\]
Then the following hold.

\medskip
\noindent
(i) Under the hypotheses of Assumption~\ref{ap214},
\[
L_\Sigma(x,y_1)-L_\Sigma(x,y_2)=\Delta_n.
\]

\medskip
\noindent
(ii) If
\[
T(x,y)\in \mathcal{B}\setminus \mathcal{B}_\varepsilon,
\]
then
\[
L_\Sigma(x,y)\ge \varepsilon n\,\Delta_n.
\]

\medskip
\noindent
(iii) Under the hypotheses of Assumption~\ref{ap29}, for every admissible backward move
$y_1^{(v,a)}$ from $y_2$,
\[
L_\Sigma(x,y_1^{(v,a)})-L_\Sigma(x,y_2)=\Delta_n,
\]
and therefore
\[
\sum_{v\in[n]:\, y_2(v)=x(v)}
\ \sum_{a\in[2]\setminus\{x(v)\}}
\exp\!\left(
-\frac{L_\Sigma(x,y_1^{(v,a)})-L_\Sigma(x,y_2)}{8}
\right)
\le n\, e^{-\Delta_n/8}.
\]
\end{proposition}

\begin{proof}
Since there is no nontrivial $\theta$-preserving permutation, the only admissible
bijection in Definition~\ref{dfb} is the identity. Thus
\[
T(x,y)\in \mathcal{B}_\varepsilon
\]
means exactly
\[
t_{1,1}(x,y)\ge n_1-n\varepsilon,
\qquad
t_{2,2}(x,y)\ge n_2-n\varepsilon.
\]
Hence, if $T(x,y)\in \mathcal{B}\setminus \mathcal{B}_\varepsilon$, then at least one of the two inequalities
fails, so
\[
D_\Omega(x,y)
=
(n_1-t_{1,1}(x,y))+(n_2-t_{2,2}(x,y))
\ge n\varepsilon.
\]
By the previous proposition,
\[
L_\Sigma(x,y)=\Delta_n D_\Omega(x,y)\ge \varepsilon n\,\Delta_n,
\]
which proves (ii).

For (i), under Assumption~\ref{ap214} the assignment $y_2$ is obtained from $y_1$ by correcting
exactly one mislabeled vertex relative to $x$. Therefore
\[
D_\Omega(x,y_1)=D_\Omega(x,y_2)+1.
\]
Using the identity $L_\Sigma(x,\cdot)=\Delta_n D_\Omega(x,\cdot)$ from the previous
proposition, we get
\[
L_\Sigma(x,y_1)-L_\Sigma(x,y_2)=\Delta_n.
\]

For (iii), each admissible backward move changes the label of exactly one correctly-labeled
vertex, hence increases $D_\Omega(x,\cdot)$ by exactly one, and therefore increases
$L_\Sigma(x,\cdot)$ by exactly $\Delta_n$. Since $k=2$, there is at most one wrong label
available for each vertex, and hence the number of admissible backward moves is at most $n$.
This gives the claimed bound.
\end{proof}

\begin{theorem}
Assume $y\in\Omega_c$ for some fixed $c\in(0,1/2)$.

\medskip
\noindent
(i) If there exists a constant $\delta>0$ such that
\[
\Delta_n\ge (8+\delta)\log n
\]
for all sufficiently large $n$, then
\[
\lim_{n\to\infty}\Pr(\hat y\in C(y))=1.
\]

\medskip
\noindent
(ii) If there exists a constant $\delta>0$ such that
\[
\Delta_n\le (8-\delta)\log n
\]
for all sufficiently large $n$, then
\[
\lim_{n\to\infty}\Pr(\hat y\in C(y))=0.
\]

\medskip
\noindent
Consequently, the exact-recovery threshold is sharp and is given by
\[
\Delta_n\sim 8\log n.
\]
Equivalently,
\[
\mu_n^2\sim (1-\rho)\log n.
\]
\end{theorem}

\begin{proof}
Since $C(y)=\{y\}$, recovery of $C(y)$ is the same as recovery of $y$ itself.

For the sufficient part, fix any
\[
\varepsilon\in\left(0,\frac{c}{3}\right).
\]
By the previous proposition,
\[
\min_{x:\,T(x,y)\in \mathcal{B}\setminus \mathcal{B}_\varepsilon} L_\Sigma(x,y)\ge \varepsilon n\,\Delta_n.
\]
Hence
\[
n\log 2-\frac{1}{8}\min_{x:\,T(x,y)\in \mathcal{B}\setminus \mathcal{B}_\varepsilon}L_\Sigma(x,y)
\le
n\log 2-\frac{\varepsilon n\,\Delta_n}{8}\to -\infty.
\]
Thus condition (\ref{ld1}) holds.

Moreover, Proposition \ref{p72}(i) shows that Assumption~\ref{ap214} holds with one-step margin
$\Delta_1=\Delta_n$. Since $\Delta_n\ge (8+\delta)\log n$, condition \eqref{ld2} holds for any
fixed $\eta\in(0,1)$ small enough. Therefore Theorem~\ref{p215} yields
\[
\Pr(\hat y\in C(y))\to 1.
\]

For the necessary part, choose any set $H_n\subset[n]$ satisfying
\[
|H_n|=\left\lfloor \frac{n}{(\log n)^2}\right\rfloor.
\]
Then
\[
\log |H_n|=(1+o(1))\log n.
\]
For each $a\in H_n$, let $y^{(a)}$ be obtained from $y$ by changing only the label of
vertex $a$ to the other label. By Proposition~\ref{p71},
\[
L_\Sigma(y^{(a)},y)=\Delta_n,\qquad \forall a\in H_n.
\]
The statistic $\eta_a$ in Theorem~\ref{mm3} depends only on the $a$-th noise column
$W_{\cdot,a}$, and the columns are independent across $a$. Therefore the Gaussian family
$\{\eta_a\}_{a\in H_n}$ is independent and
\[
\mathrm{Var}(\eta_a)=\frac{4}{\Delta_n},\qquad \forall a\in H_n.
\]
Hence its covariance matrix is
\[
\Phi_{H_n}=\frac{4}{\Delta_n}I_{|H_n|}
=
\frac{4}{\Delta_n}(I+R_n),
\qquad R_n=0.
\]
Together with \(L_\Sigma(y^{(a)},y)=\Delta_n\) and
\(\log |H_n|=(1+o(1))\log n\), Theorem~\ref{mm3} implies that, if
\(\Delta_n\le (8-\delta)\log n\), then
\[
\Pr(\hat y\in C(y))\to 0.
\]
\end{proof}

\begin{remark}
The dependence parameter $\rho$ changes the threshold through the factor $1-\rho$.
Indeed, the signal direction is $(1,-1)$, and the effective noise variance in that direction
is proportional to $1-\rho$. Positive correlation ($\rho>0$) makes recovery easier, while
negative correlation ($\rho<0$) makes recovery harder.
\end{remark}

\subsection{A cross-vertex common-factor covariance model}
In this subsection we write \(\Sigma_n\) to emphasize the dependence on \(n\); it is the
same full covariance matrix denoted by \(\Sigma\) elsewhere in the paper. Since
\(R_{\rho,n}\) is not diagonal, the covariance is not of product form
\(I_n\otimes\Sigma_0\), and whitening mixes vertex coordinates.

\begin{proposition}[Matched threshold under common-factor vertex dependence]
\label{prop:common-factor-threshold}
Assume that \(k,p\) are fixed, \(y\in\Omega_c\) for some fixed \(c\in(0,1)\), and
\[
(\mathbf A_x)_{\cdot,j}=s_n\mu_{x(j)},\qquad j\in[n],
\]
where \(s_n>0\) and \(\mu_1,\ldots,\mu_k\in\mathbb R^p\) are fixed and pairwise distinct. Let
\(C\succ0\), let \(0<\rho<1\) be fixed, and assume
\[
\operatorname{Cov}(\operatorname{vec}\mathbf W)
=
\Sigma_n
=
R_{\rho,n}\otimes C,
\qquad
R_{\rho,n}=(1-\rho)I_n+\rho\mathbf 1_n\mathbf 1_n^\top.
\]
Define
\[
d_*^2=\min_{a\ne b}(\mu_a-\mu_b)^\top C^{-1}(\mu_a-\mu_b),
\qquad
\Delta_n=\frac{s_n^2d_*^2}{1-\rho}.
\]
Then, in the unknown-community-size setting, the MLE has a sharp exact-recovery threshold:
for every fixed \(\delta>0\),
\[
\Delta_n\ge (8+\delta)\log n
\quad\Longrightarrow\quad
\Pr(\hat y\in C(y))\to1,
\]
whereas
\[
\Delta_n\le (8-\delta)\log n
\quad\Longrightarrow\quad
\Pr(\hat y\in C(y))\to0.
\]
Equivalently,
\[
s_n^2d_*^2\sim 8(1-\rho)\log n
\]
is the sharp threshold.
\end{proposition}

\begin{proof}
Throughout the proof we write
\[
\|u\|_{C^{-1}}^2:=u^\top C^{-1}u,
\qquad
d_{ab}^2:=\|\mu_a-\mu_b\|_{C^{-1}}^2,
\qquad
d_*^2:=\min_{a\ne b}d_{ab}^2 .
\]
Since \(C\succ0\) and \(0<\rho<1\), the covariance matrix
\[
\Sigma_n=R_{\rho,n}\otimes C,
\qquad
R_{\rho,n}=(1-\rho)I_n+\rho\mathbf 1_n\mathbf 1_n^\top ,
\]
is invertible. Moreover,
\[
R_{\rho,n}^{-1}
=
\frac{1}{1-\rho}I_n
-
\frac{\rho}{(1-\rho)(1-\rho+\rho n)}
\mathbf 1_n\mathbf 1_n^\top .
\]
Set
\[
\alpha:=\frac{1}{1-\rho},
\qquad
\beta_n:=\frac{\rho}{(1-\rho)(1-\rho+\rho n)}.
\]
Then
\[
R_{\rho,n}^{-1}=\alpha I_n-\beta_n\mathbf 1_n\mathbf 1_n^\top .
\]

For \(x,z\in\Omega\), define
\[
g_j(x,z):=s_n\bigl(\mu_{x(j)}-\mu_{z(j)}\bigr),
\qquad j\in[n].
\]
Using \(\Sigma_n^{-1}=R_{\rho,n}^{-1}\otimes C^{-1}\), we have the identity
\begin{align}
L_{\Sigma_n}(x,z)
&=
\sum_{i,j=1}^n
(R_{\rho,n}^{-1})_{ij}
\,
g_i(x,z)^\top C^{-1}g_j(x,z)
\nonumber\\
&=
\alpha\sum_{j=1}^n\|g_j(x,z)\|_{C^{-1}}^2
-
\beta_n
\left\|
\sum_{j=1}^n g_j(x,z)
\right\|_{C^{-1}}^2 .
\label{eq:common-factor-L}
\end{align}
This formula is the main computation.

We first verify the global separation needed for the sufficient part. Let
\[
T(x,y)=\bigl(t_{r,s}(x,y)\bigr)_{r,s\in[k]},
\qquad
t_{r,s}(x,y):=\bigl|\{j:x(j)=r,\ y(j)=s\}\bigr|.
\]
Write
\[
p_{r,s}:=\frac{t_{r,s}(x,y)}{n},
\qquad
\pi_s:=\frac{|y^{-1}(s)|}{n}.
\]
Since \(y\in\Omega_c\), we have \(\pi_s\ge c\) for every \(s\in[k]\). Define
\[
v_{r,s}:=\mu_r-\mu_s.
\]
Dividing \eqref{eq:common-factor-L} by \(n s_n^2\), we obtain
\begin{align}
\frac{L_{\Sigma_n}(x,y)}{n s_n^2}
&=
\alpha
\sum_{r,s=1}^k p_{r,s}\|v_{r,s}\|_{C^{-1}}^2
-
\beta_n n
\left\|
\sum_{r,s=1}^k p_{r,s}v_{r,s}
\right\|_{C^{-1}}^2 .
\label{eq:L-over-n}
\end{align}
Since
\[
\beta_n n
=
\alpha\frac{\rho n}{1-\rho+\rho n}
=
\alpha(1+o(1)),
\]
the right-hand side is asymptotic to
\[
\alpha
\left[
\sum_{r,s=1}^k p_{r,s}\|v_{r,s}\|_{C^{-1}}^2
-
\left\|
\sum_{r,s=1}^k p_{r,s}v_{r,s}
\right\|_{C^{-1}}^2
\right].
\]
The expression in brackets is the \(C^{-1}\)-variance of the finite set of vectors
\(\{v_{r,s}\}\) under the probability weights \(\{p_{r,s}\}\).

Fix \(\varepsilon\in(0,2c/(3k))\). Let \(\mathcal P_\varepsilon\) be the closure, in the simplex of \(k\times k\)
probability tables, of all tables \(p=(p_{r,s})\) with column sums
\[
\pi_s:=\sum_{r=1}^k p_{r,s}\ge c,\qquad s\in[k],
\]
which are outside the normalized near-truth region corresponding to
\(\mathcal B_\varepsilon\). This compact set contains every normalized overlap table
arising from assignments \(x\) with \(T(x,y)\in\mathcal B\setminus\mathcal B_\varepsilon\),
for every \(n\) and every \(y\in\Omega_c\).  Define
\[
V(p):=
\sum_{r,s=1}^k p_{r,s}\|v_{r,s}\|_{C^{-1}}^2
-
\left\|
\sum_{r,s=1}^k p_{r,s}v_{r,s}
\right\|_{C^{-1}}^2 .
\]
We claim that \(V\) is strictly positive on \(\mathcal P_\varepsilon\). Indeed, if
\(V(p)=0\), then all vectors \(v_{r,s}=\mu_r-\mu_s\) with \(p_{r,s}>0\) are equal to a
single vector \(v\). Since each true community has mass at least \(c\), for every
\(s\in[k]\) there is at least one \(r\in[k]\) such that \(p_{r,s}>0\). Hence
\[
\mu_r-\mu_s=v
\]
for at least one \(r\) for every \(s\). Thus translation by \(v\) maps the finite set
\(\{\mu_1,\ldots,\mu_k\}\) into itself. This is possible only if \(v=0\). Since the centers
are pairwise distinct, this implies \(r=s\) whenever \(p_{r,s}>0\). Therefore the table is
diagonal. Because every true community has mass at least \(c\) and
\(\varepsilon<2c/(3k)\), such a diagonal table lies in the interior of
\(\mathcal B_\varepsilon\). It therefore cannot belong to the closure
\(\mathcal P_\varepsilon\) of tables in \(\mathcal B\setminus\mathcal B_\varepsilon\), a
contradiction. Thus \(V(p)>0\) on \(\mathcal P_\varepsilon\), and compactness gives
\begin{equation}
\label{eq:common-factor-variance-lower}
\inf_{p\in\mathcal P_\varepsilon}V(p)=:\gamma_\varepsilon>0.
\end{equation}

Combining \eqref{eq:L-over-n} and \eqref{eq:common-factor-variance-lower}, for all sufficiently large \(n\),
\[
L_{\Sigma_n}(x,y)
\ge
\frac{\alpha\gamma_\varepsilon}{2}\, n s_n^2
=
\frac{\gamma_\varepsilon}{2d_*^2}\,n\Delta_n .
\]
Thus the global separation condition holds with
\[
b_\varepsilon:=\frac{\gamma_\varepsilon}{2d_*^2}>0 .
\]

We next verify the local one-step margin. Let \(x,y_1,y_2\) be an admissible triple in the
near-truth regime, where \(y_2\) is obtained from \(y_1\) by correcting one mislabeled
vertex \(v\) relative to \(x\). Thus
\[
y_2(v)=x(v),
\qquad
y_1(v)\ne x(v),
\qquad
y_1(u)=y_2(u)\quad\text{for }u\ne v .
\]
Set
\[
u:=s_n\bigl(\mu_{x(v)}-\mu_{y_1(v)}\bigr),
\qquad
S:=\sum_{j\ne v}g_j(x,y_2).
\]
Then \(g_v(x,y_1)=u\), \(g_v(x,y_2)=0\), and \(g_j(x,y_1)=g_j(x,y_2)\) for \(j\ne v\).
By \eqref{eq:common-factor-L},
\begin{align}
L_{\Sigma_n}(x,y_1)-L_{\Sigma_n}(x,y_2)
&=
(\alpha-\beta_n)\|u\|_{C^{-1}}^2
-
2\beta_n\langle u,S\rangle_{C^{-1}} .
\label{eq:local-increment}
\end{align}
Let
\[
D_*^2:=\max_{a\ne b}\|\mu_a-\mu_b\|_{C^{-1}}^2 .
\]
Since \(y_1\) is in the \(\varepsilon\)-near-truth regime relative to \(x\), the number of
mismatched vertices is at most \(k\varepsilon n\). Hence
\[
\|S\|_{C^{-1}}
\le
k\varepsilon n\, s_n D_* .
\]
Also
\[
\|u\|_{C^{-1}}^2
\ge
s_n^2 d_*^2,
\qquad
\|u\|_{C^{-1}}
\le
s_n D_* .
\]
Using \(\beta_n n=O(1)\), \eqref{eq:local-increment} gives
\begin{align}
L_{\Sigma_n}(x,y_1)-L_{\Sigma_n}(x,y_2)
&\ge
\alpha s_n^2d_*^2(1-o(1))
-
2\beta_n k\varepsilon n\,s_n^2D_*^2
\nonumber\\
&\ge
\Delta_n
\left[
1
-
C_0\varepsilon
-
o(1)
\right],
\label{eq:local-margin}
\end{align}
where \(C_0<\infty\) depends only on \(k,\rho,C\) and the centers. Since \(\varepsilon>0\)
can be chosen arbitrarily small, this is the required one-step margin at scale
\(\Delta_n\), up to an arbitrarily small multiplicative loss.

We now prove the sufficient direction. Assume
\[
\Delta_n\ge (8+\delta)\log n
\]
for some fixed \(\delta>0\). Choose \(\varepsilon>0\) sufficiently small and then choose
\(\eta\in(0,1)\) sufficiently small so that
\[
(1-C_0\varepsilon)(1-\eta)(8+\delta)>8 .
\]
The global separation estimate gives
\[
n\log k
-
\frac18
\min_{x:\,T(x,y)\in\mathcal B\setminus\mathcal B_\varepsilon}
L_{\Sigma_n}(x,y)
\le
n\log k
-
\frac{b_\varepsilon n\Delta_n}{8}(1+o(1))
\to -\infty ,
\]
because \(k\) is fixed and \(\Delta_n\to\infty\). The local estimate
\eqref{eq:local-margin} gives a one-step margin
\[
\Delta_{1,n}:=(1-C_0\varepsilon-o(1))\Delta_n .
\]
Therefore
\[
\log n+\log k-\frac{(1-\eta)\Delta_{1,n}}{8}
\to -\infty .
\]
The sufficient theorem for the unknown-size MLE then yields
\[
\Pr(\hat y\in C(y))\to1 .
\]

It remains to prove the necessary direction. Let \((a_*,b_*)\) be an ordered pair satisfying
\[
d_{a_*b_*}^2=d_*^2 .
\]
Since \(y\in\Omega_c\),
\[
|y^{-1}(b_*)|\ge cn .
\]
Choose
\[
H_n\subset y^{-1}(b_*),
\qquad
|H_n|=\left\lfloor \frac{n}{(\log n)^2}\right\rfloor .
\]
Then
\[
\log |H_n|=(1+o(1))\log n .
\]
For each \(a\in H_n\), let \(y^{(a)}\) be the single-vertex perturbation obtained from
\(y\) by changing the label of vertex \(a\) from \(b_*\) to \(a_*\). Then
\[
g_a(y^{(a)},y)=s_n(\mu_{a_*}-\mu_{b_*}),
\qquad
g_j(y^{(a)},y)=0\quad(j\ne a).
\]
Hence, by \eqref{eq:common-factor-L},
\begin{align}
L_{\Sigma_n}(y^{(a)},y)
&=
(\alpha-\beta_n)s_n^2d_*^2
\nonumber\\
&=
\Delta_n
\left(
1-\frac{\rho}{1-\rho+\rho n}
\right)
\nonumber\\
&=
\Delta_n(1+o(1)),
\label{eq:single-flip-L}
\end{align}
uniformly over \(a\in H_n\).

Let \(\eta_a\) be the local comparison statistic associated with \(y^{(a)}\), namely
\[
\eta_a
:=
\frac{
2\,
\bigl(\overrightarrow{\mathbf A_{y^{(a)}}-\mathbf A_y}\bigr)^\top
\Sigma_n^{-1}
\overrightarrow{\mathbf W}
}{
L_{\Sigma_n}(y^{(a)},y)
}.
\]
For \(a,b\in H_n\), using
\(\operatorname{Cov}(\overrightarrow{\mathbf W})=\Sigma_n\), we get
\begin{align}
\operatorname{Cov}(\eta_a,\eta_b)
&=
\frac{
4\,
\bigl(\overrightarrow{\mathbf A_{y^{(a)}}-\mathbf A_y}\bigr)^\top
\Sigma_n^{-1}
\bigl(\overrightarrow{\mathbf A_{y^{(b)}}-\mathbf A_y}\bigr)
}{
L_{\Sigma_n}(y^{(a)},y)L_{\Sigma_n}(y^{(b)},y)
}
\nonumber\\
&=
\frac{
4(R_{\rho,n}^{-1})_{ab}s_n^2d_*^2
}{
L_{\Sigma_n}(y^{(a)},y)L_{\Sigma_n}(y^{(b)},y)
}.
\label{eq:local-cov}
\end{align}
If \(a=b\), then \eqref{eq:single-flip-L} gives
\[
\operatorname{Var}(\eta_a)
=
\frac{4}{L_{\Sigma_n}(y^{(a)},y)}
=
\frac{4}{\Delta_n}(1+o(1)).
\]
If \(a\ne b\), then
\[
(R_{\rho,n}^{-1})_{ab}
=
-\beta_n,
\]
and therefore
\[
\operatorname{Cov}(\eta_a,\eta_b)
=
-\frac{4}{\Delta_n}
\frac{\rho}{1-\rho+\rho n}
(1+o(1)).
\]
Consequently the covariance matrix \(\Phi_{H_n}\) of
\(\{\eta_a:a\in H_n\}\) can be written as
\[
\Phi_{H_n}
=
\frac{4}{\Delta_n}
\left(I_{|H_n|}+\widetilde R_n\right),
\]
where
\[
\|\widetilde R_n\|_{\mathrm{op}}
\le
O\left(\frac1n\right)
+
O\left(\frac{|H_n|}{n}\right)
=
o(1),
\]
because \(|H_n|/n=(\log n)^{-2}+o(1)\). Thus the asymptotically diagonal covariance condition of Theorem~\ref{mm3} holds.

Together with
\[
L_{\Sigma_n}(y^{(a)},y)=\Delta_n(1+o(1)),
\qquad
\log |H_n|=(1+o(1))\log n,
\]
the asymptotically diagonal covariance estimate above verifies the hypotheses of
Theorem~\ref{mm3}. Therefore, if
\[
\Delta_n\le (8-\delta)\log n
\]
for some fixed \(\delta>0\), then Theorem~\ref{mm3} yields
\[
\Pr(\hat y\in C(y))\to0 .
\]

Combining the sufficient and necessary directions proves the sharp threshold
\[
\Delta_n\sim 8\log n .
\]
Since
\[
\Delta_n=\frac{s_n^2d_*^2}{1-\rho},
\]
this is equivalently
\[
s_n^2d_*^2\sim 8(1-\rho)\log n .
\]
\end{proof}

\subsection{A precision-perturbation extension}
\label{subsec:precision-perturbation}

Again we write \(\Sigma_n\) to emphasize the sequence of full covariance matrices.
The preceding examples can be extended to a broad class of models in which the
precision matrix, rather than the covariance matrix, is a perturbation of the row-independent
precision. This formulation is useful because the likelihood and the quantity
\(L_\Sigma\) are expressed directly in terms of \(\Sigma^{-1}\) when the covariance is
invertible.

Assume throughout this subsection that \(k\) and \(p\) are fixed, that
\(y\in\Omega_c\) for some fixed \(c\in(0,1)\), and that
\[
(\mathbf A_x)_{\cdot,j}=s_n\mu_{x(j)},\qquad j\in[n],
\]
where \(s_n>0\) and \(\mu_1,\ldots,\mu_k\in\mathbb R^p\) are fixed and pairwise distinct. Let
\(C\succ0\) be fixed and define
\[
\langle u,v\rangle_{C^{-1}}:=u^\top C^{-1}v,
\qquad
\|u\|_{C^{-1}}^2:=u^\top C^{-1}u.
\]
Set
\[
d_{ab}^2:=\|\mu_a-\mu_b\|_{C^{-1}}^2,
\qquad
d_*^2:=\min_{a\ne b}d_{ab}^2,
\qquad
D_*^2:=\max_{a\ne b}d_{ab}^2.
\]
Since the centers are pairwise distinct, \(0<d_*^2\le D_*^2<\infty\).

Let \(\tau_n>0\), and let
\[
M_n=I_n+E_n
\]
be a deterministic symmetric positive definite \(n\times n\) matrix. We assume that the
precision matrix of the full vectorized noise array is
\[
\Sigma_n^{-1}
=
\tau_n M_n\otimes C^{-1}.
\]
Equivalently,
\[
\Sigma_n
=
\tau_n^{-1}M_n^{-1}\otimes C.
\]
When \(E_n\) has nonzero off-diagonal entries, the noise variables attached to different
vertices are correlated, and whitening is a global operation on the vertex coordinates.

We impose the following deterministic assumptions on \(E_n\).

\medskip
\noindent
\textbf{(PP1) Uniform positivity and normalized diagonal.}
There exists a constant \(\kappa>0\) such that
\[
\lambda_{\min}(M_n)\ge \kappa
\]
for all sufficiently large \(n\), and
\[
\max_{i\in[n]} |(E_n)_{ii}|\to0.
\]

\medskip
\noindent
\textbf{(PP2) Small local row mass.}
For \(\varepsilon>0\), define
\[
r_n(\varepsilon)
:=
\max_{i\in[n]}
\max_{\substack{S\subset[n]\setminus\{i\}\\ |S|\le k\varepsilon n}}
\sum_{j\in S}|(E_n)_{ij}|.
\]
Assume that
\[
\lim_{\varepsilon\downarrow0}\limsup_{n\to\infty}r_n(\varepsilon)=0.
\]

\medskip
\noindent
\textbf{(PP3) A nearly diagonal local testing family.}
There exists an ordered pair \((a_*,b_*)\) such that
\[
d_{a_*b_*}^2=d_*^2,
\]
and there exists a subset
\[
H_n\subset y^{-1}(b_*)
\]
such that
\[
\log |H_n|=(1+o(1))\log n
\]
and
\[
\|(E_n)_{H_n,H_n}\|_{\mathrm{op}}\to0.
\]

Define the effective separation scale
\[
\Delta_n:=\tau_n s_n^2 d_*^2.
\]

\begin{proposition}[Matched threshold under precision perturbations]
\label{prop:precision-perturbation-threshold}
Under assumptions \emph{(PP1)--(PP3)}, in the unknown-community-size setting, the MLE
has the following sharp exact-recovery threshold. For every fixed \(\delta>0\),
\[
\Delta_n\ge (8+\delta)\log n
\quad\Longrightarrow\quad
\Pr(\hat y\in C(y))\to1,
\]
whereas
\[
\Delta_n\le (8-\delta)\log n
\quad\Longrightarrow\quad
\Pr(\hat y\in C(y))\to0.
\]
Equivalently,
\[
\tau_n s_n^2d_*^2\sim 8\log n
\]
is the sharp threshold.
\end{proposition}

\begin{proof}
For \(x,z\in\Omega\), define
\[
g_j(x,z):=s_n\bigl(\mu_{x(j)}-\mu_{z(j)}\bigr),
\qquad j\in[n].
\]
Since
\[
\Sigma_n^{-1}=\tau_nM_n\otimes C^{-1},
\]
we have the exact identity
\begin{align}
L_{\Sigma_n}(x,z)
&=
\tau_n
\sum_{i,j=1}^n
(M_n)_{ij}
\langle g_i(x,z),g_j(x,z)\rangle_{C^{-1}}.
\label{eq:precision-perturbation-L}
\end{align}

We first prove the sufficient part. Let
\[
m(x,y):=\bigl|\{j\in[n]:x(j)\ne y(j)\}\bigr|
\]
be the number of mismatched vertices. Since the centers are pairwise distinct, the only
\(\theta\)-preserving relabeling in this example is the identity. Hence if
\[
T(x,y)\in\mathcal B\setminus\mathcal B_\varepsilon,
\]
then
\[
m(x,y)\ge \varepsilon n.
\]
By \emph{(PP1)} and \eqref{eq:precision-perturbation-L},
\begin{align}
L_{\Sigma_n}(x,y)
&\ge
\tau_n\kappa
\sum_{j=1}^n\|g_j(x,y)\|_{C^{-1}}^2
\nonumber\\
&\ge
\tau_n\kappa\,m(x,y)\,s_n^2d_*^2
\nonumber\\
&\ge
\kappa\varepsilon n\,\Delta_n .
\label{eq:precision-global}
\end{align}
Thus the global separation condition holds with
\[
b_\varepsilon=\kappa\varepsilon.
\]

Next we verify the near-truth one-step margin. Let \((x,y_1,y_2)\) be an admissible
triple in the sense of Assumption~\ref{ap214}. Thus \(y_2\) is obtained from \(y_1\) by
correcting one mislabeled vertex, say \(v\), relative to \(x\):
\[
y_2(v)=x(v),
\qquad
y_1(v)\ne x(v),
\qquad
y_1(u)=y_2(u)\quad\text{for }u\ne v.
\]
Set
\[
u:=s_n\bigl(\mu_{x(v)}-\mu_{y_1(v)}\bigr)
\]
and
\[
S:=\{j\ne v:g_j(x,y_2)\ne0\}.
\]
In the \(\varepsilon\)-near-truth regime, the number of mismatched vertices is at most
\(k\varepsilon n\), and hence
\[
|S|\le k\varepsilon n.
\]
Since
\[
g_v(x,y_1)=u,\qquad g_v(x,y_2)=0,
\]
and
\[
g_j(x,y_1)=g_j(x,y_2)\qquad (j\ne v),
\]
we get from \eqref{eq:precision-perturbation-L} that
\begin{align}
L_{\Sigma_n}(x,y_1)-L_{\Sigma_n}(x,y_2)
&=
\tau_n(M_n)_{vv}\|u\|_{C^{-1}}^2
+
2\tau_n
\sum_{j\in S}
(M_n)_{vj}
\langle u,g_j(x,y_2)\rangle_{C^{-1}} .
\label{eq:precision-local-exact}
\end{align}
By \emph{(PP1)},
\[
(M_n)_{vv}=1+o(1)
\]
uniformly in \(v\). Moreover,
\[
\|u\|_{C^{-1}}^2\ge s_n^2d_*^2,
\qquad
\|u\|_{C^{-1}}\le s_nD_*,
\]
and, for every \(j\in S\),
\[
\|g_j(x,y_2)\|_{C^{-1}}\le s_nD_*.
\]
Since \((M_n)_{vj}=(E_n)_{vj}\) for \(j\ne v\), \emph{(PP2)} gives
\[
\sum_{j\in S}|(M_n)_{vj}|
\le r_n(\varepsilon).
\]
Therefore \eqref{eq:precision-local-exact} implies
\begin{align}
L_{\Sigma_n}(x,y_1)-L_{\Sigma_n}(x,y_2)
&\ge
\tau_n s_n^2d_*^2(1-o(1))
-
2\tau_n s_n^2D_*^2 r_n(\varepsilon)
\nonumber\\
&=
\Delta_n
\left[
1
-
2\frac{D_*^2}{d_*^2}r_n(\varepsilon)
-
o(1)
\right].
\label{eq:precision-local-margin}
\end{align}
Let
\[
a_\varepsilon
:=
2\frac{D_*^2}{d_*^2}
\limsup_{n\to\infty}r_n(\varepsilon).
\]
By \emph{(PP2)},
\[
a_\varepsilon\to0
\qquad\text{as }\varepsilon\downarrow0.
\]
Hence the one-step margin is at least
\[
(1-a_\varepsilon-o(1))\Delta_n.
\]

Assume now that
\[
\Delta_n\ge(8+\delta)\log n
\]
for some fixed \(\delta>0\). Choose \(\varepsilon>0\) sufficiently small and then choose
\(\eta\in(0,1)\) sufficiently small so that
\[
(1-a_\varepsilon)(1-\eta)(8+\delta)>8.
\]
The global estimate \eqref{eq:precision-global} gives
\[
n\log k
-
\frac18
\min_{x:\,T(x,y)\in\mathcal B\setminus\mathcal B_\varepsilon}
L_{\Sigma_n}(x,y)
\le
n\log k-\frac{\kappa\varepsilon n\Delta_n}{8}
\to -\infty,
\]
because \(k\) is fixed and \(\Delta_n\to\infty\). The local estimate
\eqref{eq:precision-local-margin} gives a one-step margin
\[
\Delta_{1,n}:=(1-a_\varepsilon-o(1))\Delta_n,
\]
and hence
\[
\log n+\log k-\frac{(1-\eta)\Delta_{1,n}}8\to-\infty.
\]
The sufficient theorem for the unknown-size MLE, Theorem~\ref{p215}, therefore yields
\[
\Pr(\hat y\in C(y))\to1.
\]

We now prove the necessary part. Let \((a_*,b_*)\) and \(H_n\) be as in \emph{(PP3)}.
For each \(a\in H_n\), let \(y^{(a)}\) be the single-vertex perturbation obtained from
\(y\) by changing the label of vertex \(a\) from \(b_*\) to \(a_*\). Then
\[
g_a(y^{(a)},y)
=
s_n(\mu_{a_*}-\mu_{b_*}),
\qquad
g_j(y^{(a)},y)=0\quad(j\ne a).
\]
Using \eqref{eq:precision-perturbation-L}, we get
\begin{align}
L_{\Sigma_n}(y^{(a)},y)
&=
\tau_n(M_n)_{aa}s_n^2d_*^2
\nonumber\\
&=
(M_n)_{aa}\Delta_n
\nonumber\\
&=
\Delta_n(1+o(1)),
\label{eq:precision-single-flip}
\end{align}
uniformly in \(a\in H_n\), where the last step follows from \emph{(PP3)} and \emph{(PP1)}.

Let \(\eta_a\) be the Gaussian comparison statistic associated with \(y^{(a)}\), namely
\[
\eta_a
:=
\frac{
2\,
\bigl(\overrightarrow{\mathbf A_{y^{(a)}}-\mathbf A_y}\bigr)^\top
\Sigma_n^{-1}
\overrightarrow{\mathbf W}
}{
L_{\Sigma_n}(y^{(a)},y)
}.
\]
For \(a,b\in H_n\), since
\(\operatorname{Cov}(\overrightarrow{\mathbf W})=\Sigma_n\), we have
\begin{align}
\operatorname{Cov}(\eta_a,\eta_b)
&=
\frac{
4
\bigl(\overrightarrow{\mathbf A_{y^{(a)}}-\mathbf A_y}\bigr)^\top
\Sigma_n^{-1}
\bigl(\overrightarrow{\mathbf A_{y^{(b)}}-\mathbf A_y}\bigr)
}{
L_{\Sigma_n}(y^{(a)},y)L_{\Sigma_n}(y^{(b)},y)
}
\nonumber\\
&=
\frac{
4\tau_n(M_n)_{ab}s_n^2d_*^2
}{
L_{\Sigma_n}(y^{(a)},y)L_{\Sigma_n}(y^{(b)},y)
}.
\label{eq:precision-local-cov}
\end{align}
Let
\[
D_{H_n}:=\mathrm{diag}\bigl((M_n)_{aa}:a\in H_n\bigr).
\]
By \eqref{eq:precision-single-flip} and \eqref{eq:precision-local-cov}, the covariance matrix
\(\Phi_{H_n}\) of \(\{\eta_a:a\in H_n\}\) satisfies
\[
\Phi_{H_n}
=
\frac4{\Delta_n}
D_{H_n}^{-1}
(M_n)_{H_n,H_n}
D_{H_n}^{-1}.
\]
By \emph{(PP3)},
\[
(M_n)_{H_n,H_n}
=
I_{|H_n|}+(E_n)_{H_n,H_n},
\qquad
\|(E_n)_{H_n,H_n}\|_{\mathrm{op}}\to0.
\]
In particular,
\[
D_{H_n}=I_{|H_n|}+o(1)
\]
in operator norm. Hence
\[
\Phi_{H_n}
=
\frac4{\Delta_n}
\left(I_{|H_n|}+R_n\right),
\qquad
\|R_n\|_{\mathrm{op}}\to0.
\]
Thus the asymptotically diagonal covariance condition of Theorem~\ref{mm3} holds.

Together with
\[
L_{\Sigma_n}(y^{(a)},y)=\Delta_n(1+o(1)),
\qquad
\log |H_n|=(1+o(1))\log n,
\]
the asymptotically diagonal covariance estimate above verifies the hypotheses of
Theorem~\ref{mm3}. Therefore, if
\[
\Delta_n\le(8-\delta)\log n
\]
for some fixed \(\delta>0\), then Theorem~\ref{mm3} yields
\[
\Pr(\hat y\in C(y))\to0.
\]

Combining the sufficient and necessary directions proves the sharp threshold
\[
\Delta_n\sim8\log n.
\]
Since
\[
\Delta_n=\tau_ns_n^2d_*^2,
\]
the equivalent form of the threshold is
\[
\tau_ns_n^2d_*^2\sim8\log n.
\]
\end{proof}

\begin{remark}[Examples covered by Proposition~\ref{prop:precision-perturbation-threshold}]
The assumptions above allow many non-row-independent covariance structures. For example,
let \(v_n=n^{-1/2}\mathbf1_n\), let \(\alpha>-1\) be fixed, and take
\[
M_n=I_n+\alpha v_nv_n^\top.
\]
Then \(M_n\) is uniformly positive definite, its off-diagonal entries are of order \(1/n\),
and the whitening operation mixes all vertices. More generally, delocalized low-rank
perturbations
\[
M_n=I_n+U_nB U_n^\top
\]
with fixed rank, bounded \(B\), incoherent rows
\[
\max_i\|(U_n)_{i,\cdot}\|=O(n^{-1/2}),
\]
and the uniform positivity condition
\[
\lambda_{\min}(I_n+U_nB U_n^\top)\ge \kappa>0
\]
satisfy the local row-mass and local covariance conditions above. Thus the model is not a row-independent Gaussian mixture model after
whitening, even though the exact-recovery threshold is still governed by the single local
precision scale \(\Delta_n=\tau_ns_n^2d_*^2\).
\end{remark}

\subsection{Relation to the compact no-gap criterion}

The examples above are intended to be read through the compact criterion in
Proposition~\ref{prop:compact-sharp-threshold}.  In the row-independent benchmark,
the local comparison statistics are exactly independent after the product whitening.  In
the common-factor and precision-perturbation examples, whitening is global and mixes the
vertex coordinates, but the dangerous local comparison statistics are still asymptotically
diagonal on a large testing family.  In all three cases the same deterministic scale
\(\Delta_n\) controls the global separation, the local correction margin, and the covariance
of the local perturbation statistics.  This is the no-gap mechanism behind the sharp
threshold \(\Delta_n\sim 8\log n\).
\section{Concluding remarks and future directions}\label{sec:conclusion}

\textbf{Summary of the main message.}
We studied exact recovery for community detection in a Gaussian mixture model in which
the Gaussian perturbation may be dependent and heterogeneous. The covariance matrix
$\Sigma$ is allowed to be non-diagonal and, in the general formulation, possibly singular.
In the singular case, the Gaussian likelihood must be interpreted on the support of the
induced Gaussian measure, and the maximum likelihood estimator is formulated as a
constrained optimization problem in which the Moore--Penrose inverse $\Sigma^\dagger$
appears in the quadratic part of the likelihood. Across both the singular and invertible
settings, the comparison between two assignments is governed by the $\Sigma$-whitened
separation
\[
L_\Sigma(x,y)
=
(\overrightarrow{\mathbf{A}_x-\mathbf{A}_y})^t\Sigma^\dagger\overrightarrow{\mathbf{A}_x-\mathbf{A}_y},
\]
which provides a unified geometric quantity controlling recoverability.

\medskip
\textbf{Statistical implications.}
For unknown community sizes, Theorem~\ref{p215} gives sufficient conditions for exact recovery
by decomposing competing assignments into a ``globally different'' regime, controlled by
a uniform lower bound on $L_\Sigma(x,y)$, and an ``aligned near-truth'' regime, controlled
through local one-vertex moves. For known community sizes, Theorem~\ref{m27} yields an
analogous sufficient condition. On the converse side, under the additional assumption that
$\Sigma$ is invertible, Theorems~\ref{mm3} and~\ref{p31} show that exact recovery fails when there
exists a sufficiently large family of local perturbations whose associated Gaussian
comparison statistics have asymptotically diagonal covariance. Thus the paper separates
a general sufficiency theory, valid for possibly singular covariance structures, from an
invertible-covariance converse theory. Together, these results clarify how dependence in
the noise affects recovery through the interaction between $L_\Sigma(\cdot,\cdot)$ and the
geometry induced by the covariance structure.

\medskip
\textbf{Examples beyond row-wise whitening and the no-gap mechanism.}
Section~\ref{sect:ex} first revisits the row-independent non-diagonal block-covariance
model as a benchmark, showing that the general likelihood criterion recovers the
classical whitened threshold. The main examples, however, are the common-factor
vertex-covariance model and the precision-perturbation model. In both cases, the
covariance is not of the product form $I_n\otimes\Sigma_0$, and whitening mixes vertex
coordinates rather than preserving a row-independent Gaussian-mixture structure. For
these cross-vertex dependent models, the sufficient and necessary conditions match at a
single deterministic scale $\Delta_n$, yielding the sharp threshold
$\Delta_n\sim8\log n$. These examples demonstrate the no-gap mechanism isolated in
Proposition~\ref{prop:compact-sharp-threshold}: global separation, local correction
margins, and an asymptotically diagonal family of local comparison statistics together
determine the exact-recovery boundary.

\medskip
\textbf{Directions for further research.}
The present work suggests several natural directions.

\begin{enumerate}
\item \textbf{Unknown or partially known covariance.}
In many applications $\Sigma$ is not known a priori.
It would be important to understand (i) when $\Sigma$ (or $\Sigma^\dagger$) can be
consistently estimated from the data under structural assumptions (e.g.\ sparsity,
low-rank, block structure, or dependence driven by vertex features), and (ii) how the error
in estimating $\Sigma^\dagger$ propagates into the recovery guarantees for the MLE.
Developing recovery conditions that are robust to covariance misspecification is also of
interest.

\item \textbf{Computationally efficient estimators under dependence.}
While the MLE characterizes the statistical limit, computing it is generally intractable
for large $n$.
A key open direction is to design and analyze polynomial-time procedures that incorporate
the whitening induced by $\Sigma^\dagger$ (e.g.\ convex relaxations or spectral methods in a
whitened inner product), and to determine whether there is a computational--statistical gap
in dependent settings.

\item \textbf{Beyond exact recovery and beyond Gaussianity.}
It would be natural to extend the analysis to almost exact recovery, minimax
misclassification rates, and finite-sample bounds.
Another direction is to replace Gaussian noise by broader families (e.g.\ sub-Gaussian,
heavy-tailed, or contaminated models) and to allow missing data.
Understanding which parts of the theory depend essentially on Gaussian comparison
inequalities and which extend to more general noise remains an interesting question.
\end{enumerate}

\bigskip
\bigskip
\noindent{\textbf{Acknowledgements.}} ZL's research is supported by National Science Foundation grant 1608896 and Simons Foundation grant 638143.  
\bibliography{references}
\bibliographystyle{plain}

\end{document}